\documentclass{article}
\usepackage{amsmath, geometry,amssymb,amscd, float, verbatim, latexsym,subfig,wrapfig}
\usepackage{setspace}
\usepackage[pdftex]{graphicx}
\usepackage[all]{xy}
\usepackage{lscape,mathtools}
\usepackage{cite}
\usepackage{amsthm}
\usepackage{color}

\newtheorem{theorem}{Theorem}[section]
\newtheorem{lemma}[theorem]{Lemma}
\newtheorem{definition}[theorem]{Definition}
\newtheorem{proposition}[theorem]{Proposition}
\newtheorem{corollary}[theorem]{Corollary}
\newtheorem{conjecture}[theorem]{Conjecture}

\newtheorem{algorithm}[theorem]{Algorithm}

\newtheorem{remark}[theorem]{Remark}

\newcommand{\R}{\mathbb{R}}
\newcommand{\C}{\mathbb{C}}

\newcommand{\N}{\mathbb{N}}
\newcommand{\Z}{\mathbb{Z}}

\newcommand{\pp}{\tfrac{\pi}{2}}

\newcommand{\cA}{\mathcal{A}}

\newcommand{\cS}{\mathcal{S}}
\newcommand{\cX}{\mathcal{X}}

\newcommand{\bydef}{\,\stackrel{\mbox{\tiny\textnormal{\raisebox{0ex}[0ex][0ex]{def}}}}{=}\,} 

\title{Stability and  Uniqueness of Slowly Oscillating Periodic Solutions to Wright's Equation}

\author{
Jonathan Jaquette \thanks{Partially supported by NSF DMS 0915019,	NSF DMS 1248071} \thanks{
Department of Mathematics, Hill Center-Busch Campus, Rutgers, The State University of New Jersey, Piscataway, NJ, USA, 08854-8019.
{\tt jaquette@math.rutgers.edu}
}
\and Jean-Philippe Lessard \thanks{Partially supported by NSERC} \thanks
        { Universit\'e Laval, D\'epartement de Math\'ematiques et de Statistique, 1045 avenue de la M\'edecine, Qu\'ebec, QC, G1V0A6, Canada.  
        {\tt jean-philippe.lessard@mat.ulaval.ca}
        }
         \and Konstantin Mischaikow\thanks{Partially supported by NSF DMS grants 0915019, 1125174, 1248071, 1521771, 1622401}\thanks{
         Department of Mathematics, Hill Center-Busch Campus, Rutgers, The State University of New Jersey, Piscataway, NJ, USA, 08854-8019.
{\tt mischaik@math.rutgers.edu}}
}

\begin{document}

\maketitle

\begin{abstract}
In this paper, we prove that Wright's equation $y'(t) = - \alpha y(t-1) \{1 + y(t)\}$ has a unique slowly oscillating periodic solution (SOPS) for all parameter values $\alpha \in [ 1.9,6.0]$, up to time translation. 
Our proof is based on a same strategy employed earlier by Xie \cite{xie1993uniqueness}; show that every SOPS is asymptotically stable. 
We first introduce a {\em branch and bound} algorithm to control all SOPS using {\em bounding functions} at all parameter values $\alpha \in [ 1.9,6.0]$. 
Once the bounding functions are constructed, we then control the Floquet multipliers of all possible SOPS by solving rigorously an eigenvalue problem, again using a formulation introduced by Xie. 
Using these two main steps, we prove that all SOPS of Wright's equation are asymptotically stable for $\alpha \in [ 1.9,6.0]$, and the proof follows. 
This result is a step toward the proof of the Jones' Conjecture formulated in 1962. 
\end{abstract}

\begin{center}
{\bf \small Key words.} 
{ \small Wright's Equation $\cdot$ Jones's Conjecture $\cdot$ Delay Differential
Equations \\ Computer-Assisted Proofs $\cdot$ Branch and Bound}
\end{center}


\section{Introduction}

In \cite{wright1955non} Wright studied the delay differential equation (DDE)
\begin{equation}
	\label{eq:Wright}
	y'(t) = - \alpha y(t-1) \{1 + y(t)\}
\end{equation}
and showed that if $ \alpha > \tfrac{\pi }{2} $ and a solution  $ y(t) $ was positive  for $ t \in (0,1)$, then $y$ does not approach $0$ as $ t \to \infty$ and there are infinitely many zeros of $y$ separated by a distance greater than unity. 
If a periodic solution has this property it is said to be ``slowly oscillating'' and is formally defined as follows.  
\begin{definition}
	A \emph{slowly oscillating periodic solution (SOPS)}  is a periodic solution $y(t)$ with the following property: 
	there exist $q, \bar{q} >1$ and $ L = q + \bar{q}$ such that up to a time translation, 
	$ y(t) >0$ on the interval $(0,q)$,
	$y(t) < 0$ on the interval $(q,L)$, 
	and $y(t+L) = y(t)$ for all $ t$,
	so that $ L $ is the minimal period of $y(t)$.  
\end{definition}

The existence of SOPS to~\eqref{eq:Wright}  for all $ \alpha > \pp$  was proven in 1962 by Jones \cite{jones1962existence} who formulated the following conjecture based on  numerical experiments \cite{jones1962nonlinear}: 
\begin{conjecture}[Jones' Conjecture] \label{Jones Conjecture}
	For all $ \alpha > \pp$ there exists a unique SOPS to \eqref{eq:Wright} (up to time translation).
\end{conjecture}

We briefly describe  results on the global dynamics of Wright's equation. 
At $\alpha = \pp$ there is a super-critical Hopf bifurcation  \cite{chow1977integral}. 
This branch of SOPS grows without bound, that is  for each $L > 4$ this branch contains SOPS of period $L$  \cite{nussbaum1977range} 
and for all $ \alpha > \pp$ this branch contains SOPS at parameter $ \alpha $   \cite{nussbaum1975global}. 
For $\alpha > \pp$ there is an asymptotically stable annulus in the $( x(t), x(t-1) ) $ plane whose boundary is a pair of slowly oscillating periodic orbits \cite{kaplan1975stability}.  
In \cite{xie1991thesis} Xie showed that if $\alpha \geq 5.67$,  then there is a unique SOPS to Wright's equation. 
In this paper using the computer we are able to extend Xie's method of proof thereby obtaining the following result.
\begin{theorem}
	There exists a unique SOPS to Wright's equation for $ \alpha \in [ 1.9,6.0]$. 
	\label{prop:MainResult}
\end{theorem}
\noindent
Combining Theorem \ref{prop:MainResult} with the work of \cite{xie1991thesis} it follows that 
there exists a unique SOPS to Wright's equation for $ \alpha \geq 1.9$.

The set of all periodic orbits to Wright's equation for $ \alpha>0$ form a $2$-dimensional manifold \cite{regala1989periodic}.
Based on this, it is proposed in \cite{lessard2010recent} to divide Conjecture \ref{Jones Conjecture} into two parts: 
\begin{enumerate}
\item[(1)] there are no saddle-node bifurcations in the branch of SOPS emanating from the Hopf bifurcation at $\alpha = \pp$, and 
\item[(2)] there are no other connected components (isolas) of SOPS for $ \alpha > \pp$. 
\end{enumerate}
Using computer-assisted proofs, it has been shown that  saddle-node bifurcations do not occur for neither $\alpha \in [\pp + \epsilon_1, 2.3]$ where $\epsilon_1 = 7.3165 \times 10^{-4}$ \cite{lessard2010recent} nor $ \alpha \in ( \pp , \pp+ \epsilon_2)$ for $ \epsilon_2 = 6.830 \times 10^{-3} $ \cite{BergJaquette}. 
Together with these results, Theorem \ref{prop:MainResult} fills in the gap $\alpha \in ( 2.3,5.67)$ needed to resolve part (1) of the Jones' conjecture (see \cite[Corollary 4.8]{BergJaquette} for the complete proof).  
Furthermore,  the paper \cite{BergJaquette} explicitly constructs a neighborhood  about the bifurcation point at $ \alpha = \pp$ within which there exists at most one SOPS to Wright's equation. 
The proof for Theorem \ref{prop:MainResult} is given in Section \ref{sec:Discussion}. 
Thereby the Jones conjecture (Conjecture~\ref{Jones Conjecture}) is reduced to the following open problem:
\begin{conjecture}
	There are no isolas of SOPS for $ \alpha \in ( \pp , 1.9)$. 
	\label{prop:RemainingConjecture}
\end{conjecture}

To study Wright's equation, we make the change of variables $ x = \ln (1+y)$, obtaining the equivalent   differential equation 
\begin{equation}
x'(t) = - \alpha f( x(t-1)) 
\label{eq:FeedBack}
\end{equation}
where $f(x) = e^x -1$, which we will hereafter refer to as Wright's equation. 
Critical to Xie's result on the Jones conjecture is the relation between the asymptotic dynamics of SOPS and their global uniqueness.
\begin{theorem}[See \cite{xie1991thesis,xie1993uniqueness}]
	If $ \alpha > \pp$ and every SOPS to~\eqref{eq:FeedBack} is asymptotically stable, then~\eqref{eq:FeedBack} has a unique SOPS up to a time translation.  
	\label{prop:Xie}
\end{theorem}
Following Xie's approach, we define the function space
\[
\cX \bydef  \left\{ x \in C^1 ( \R,\R) \mid x(0) = 0, x'(0) > 0 \mbox{ and } x(t) <0 \mbox{ for } t \in (-1,0)     \right\}.
\]
Up to a time translation, the space $ \cX$ contains all SOPS to Wright's equation. 
Xie showed that if $ x \in \cX$ is a SOPS to Wright's equation with period $L$, then its nontrivial Floquet multipliers $\lambda \in \C$ are given by solutions to the nonautonomous linear  DDE:
	\begin{equation}
	y'(t) = - \alpha f'(x(t-1)) y(t-1) \label{eq:LinearizedIntro} 
	\end{equation}
subject to the boundary condition
\begin{equation}
\lambda y(s) =  - y(L) \frac{x'(s+L)}{x'(L)} + y(s+L), \qquad  s \in [-1,0].
\label{eq:MultiplierIntro}
\end{equation}
For a SOPS $  x \in \cX$, showing that $ |\lambda | < 1$ for all possible solutions $ y$ to \eqref{eq:LinearizedIntro} and \eqref{eq:MultiplierIntro} it suffices to show that $x$ is asymptotically stable. 
By doing so for all possible SOPS to Wright's equation when $ \alpha \geq 5.67$, Xie achieved his proof for uniqueness. 
Xie's method has two parts: (1) obtain estimates on SOPS to Wright's equation and (2) use these estimates to develop 
an upper bound on the magnitude of their Floquet multipliers.  
Xie was only able to obtain a proof for $\alpha \geq 5.67$ because of the difficulty of the first part. 
In this paper we continue Xie's method by means of a computer-assisted proof.

Our approach to obtaining bounds on SOPS is based on an algorithmic case-by-case analysis of the locations of the zeros of a function $  x \in \cX$ and the size of its extrema. 
In \cite[Lemmas 4 and 5]{wright1955non} it is shown that if $ x \in \cX$  and $ \alpha > 1$ then the zeros $ \{ z_i ( x) \}_{i=0}^{\infty}$ of $ x$ are countably infinite and $ z_{i+1}(x) - z_{i}(x) >1$.  
This result implies that we can define the maps $q:\mathcal{X} \to (1,\infty)$ and $\bar q:\mathcal{X} \to (1,\infty)$ 
as follows given $x \in \mathcal{X}$:
\begin{align*}
	q(x) &\bydef z_1(x) -z_0(x),\\
	\bar{q}(x) &\bydef z_2(x) - z_1(x).
\end{align*}
By construction, if $ x \in \cX$, then its first zero is $ z_0(x) = 0$. 
Moreover, if $x$ is a SOPS then $ q(x) + \bar{q}(x)$ is its period and furthermore  if it solves  \eqref{eq:FeedBack}, then its extrema are given as 
\begin{align*}
	\max_{t \in \R} x(t) &= x(1), \\
	\min_{t \in \R} x(t) &= x\left( q(x) + 1 \right).
\end{align*}

In \cite{neumaier2014global} a branch and bound algorithm is applied to the 2-dimensional domain $\{ \max x , \min x\}$ 
to show that there do not exist any SOPS to Wright's equation for $ \alpha \leq 1.5706$, making substantial progress on Wright's conjecture that the origin is the global attractor to~\eqref{eq:Wright} for $ \alpha < \pp \approx 1.57079$.   
Without an exact value for $ q(x)$, one cannot  pinpoint the location of the minimum of $ x$. 
To account for this ambiguity, the authors in \cite{neumaier2014global} use a  collection of six different functions to bound $x$, each defined relative to one of the zeros $\{ z_0(x), z_1(x) , z_2(x) \}$.
We use an alternative approach that allows us to work with just two bounding functions. 
In particular, we classify the space $ \cX$ according to the finite dimensional reduction map  $ \kappa : \cX \to \R^3$ defined as follows:
\begin{equation}
 \kappa( x) \bydef \{ q(x), \bar{q}(x) , x(1) \} .
\label{eq:Projection}
\end{equation}
Relative to a SOPS's image under $\kappa$, we formally define bounding functions as follows.
\begin{definition}
	Fix an interval $I_\alpha = [ \alpha_{min} , \alpha_{max}]$ and a region $ K \subset \R^3$. The functions $ \ell_K,u_K:\R\to \R$ are   \emph{bounding functions} (associated with $K$) if 
	\[
	\ell_K(t) \leq x(t) \leq u_K(t), \quad \text{for all } t \in \R, 
	\]
	whenever $x \in \cX$ is a SOPS to Wright's equation at a parameter $ \alpha \in I_\alpha$ satisfying $ \kappa(x) \in K$. 
\end{definition}

In practice, we define the functions $ u_K$, $\ell_K$ as  piecewise constant functions, which are easy to represent and rigorously integrate on a computer.  
To ensure proper mathematical rigor and computational reliability, we have used interval arithmetic for the execution of our computer-assisted proofs \cite{rump1999intlab,moore1966interval}. 
Notably, our algorithms use a rigorous numerical integrator for delay differential equations, 
about which there is a growing literature \cite{neumaier2014global,minamoto2010numerical,szczelina2014rigorous,szczelina2016algorithm}.      
These computational details are discussed further in Appendix \ref{sec:Appendix}. 

To summarize by Theorem \ref{prop:Xie}, in order to prove that there is a unique SOPS, it is sufficient to show that every SOPS is asymptotically stable.  
This breaks into two major parts: characterizing SOPS to Wright's equation and bounding their Floquet multipliers.  
To accomplish the first part,  we begin by constructing  compact regions $ K_1, K_2  \subset \R^3$, described in  Algorithm \ref{alg:InitialBoundsleq3} and Algorithm \ref{alg:InitialBoundsgeq3} respectively, for which $K_1 \cup K_2$ contains the $\kappa$-image of all SOPS to Wright's equation. 
We then use a \emph{branch and prune} method, defined in Algorithm \ref{alg:BranchAndPrune}, to refine these initial global bounds. 
This algorithm \emph{branches} by subdividing $ K_1 \cup K_2 $ into smaller pieces, and \emph{prunes} by using Algorithm \ref{alg:Prune} to develop tighter bounding functions.  
The end result of this process is a collection $\cA$ of subsets of $K \subset \R^3$, and in Theorem \ref{prop:BranchNPrune} we prove for a given parameter range $ [\alpha_{min} , \alpha_{max}]$ that if $ x \in \cX$ is a SOPS then $ \kappa(x) \in \bigcup_{K \in \cA} K$. 
The task then becomes to show that every SOPS is asymptotically stable.  
For a given region $K \subset \R^3$, we use Algorithm \ref{alg:FloquetBound}  to derive a bound on the Floquet multipliers of any SOPS with $\kappa$-image contained in $K$. 
This is then combined with the branch and prune method in Algorithm \ref{alg:Main}. 
Finally, the proof to Theorem \ref{prop:MainResult} is given in Section 6, where, in addition, we discuss the computational limitations of our approach.

\section{A computational approach }
\label{sec:BrandAndPrune}

Theorem \ref{prop:Xie} effectively transforms Jones's Conjecture (Conjecture~\ref{Jones Conjecture}) into the problem of studying the asymptotic dynamics of  SOPS, and in turn, their Floquet multipliers / Lyapunov exponents.
In a neighborhood about a periodic function, one can develop estimates on these Floquet multipliers \cite{castelli2013rigorous,xie1991thesis}. 
However these bounds rely significantly on this  neighborhood about the periodic function being relatively small.  
In effect, Xie shows that any SOPS to Wright's equation is stable for each $ \alpha \geq 5.67$ by first showing that all such solutions reside within a narrow region, and subsequently shows that all periodic orbits in that region are asymptotically stable.  
This first step is the more difficult part, and the reason Xie restricts his proof to $ \alpha  \geq  5.67$.

In Xie's thesis \cite{xie1991thesis} a case-by-case analysis is used to obtain   a region within which all SOPS must lie.  
Specifically, if $\bar{q}(x) \geq  3$ then  asymptotic analysis \cite{nussbaum1982asymptotic} precisely describes the approximate form of the SOPS with tight error estimates. 
For the alternative case, Xie divided the possibility of $ \bar{q}(x) <3$ into several sub-cases and showed that each of these led to a contradiction when $ \alpha \geq 5.67$.   
In our analysis we   make similar assumptions by considering a SOPS's image under the map $ \kappa(x) = \{ q(x), \bar{q}(x) , x(1) \}$ and the bounding functions associated with various regions $K \subset \R^3$. 
For any region $ K \subset \R^3$ there is not a unique choice of bounding functions.  
In fact, we develop techniques which iteratively tighten the bounding functions for a fixed region $K$. 
If in our process of tightening bounding functions we derive a contradiction, such as $\ell_K(t) > u_K(t)$, then we may conclude that there does not exist any SOPS $x $ for which $ \kappa(x) \in K$.

In performing a case-by-case analysis of SOPS to Wright's equation, we are principally concerned with bounding all possible SOPS, and  we find it useful to introduce the notion of an $I_\alpha$-exhaustive set. 
\begin{definition}
	\label{def:IalphaEx}
	Fix an interval $I_\alpha = [ \alpha_{min} , \alpha_{max}]$ 
	and consider a set $K  \subset \R^3$. 
	The set $K$ is   \emph{$I_\alpha$-exhaustive} if $ \kappa(x) \in K$ for any SOPS $x \in \cX$ to Wright's equation at parameter $ \alpha \in I_\alpha$. 
\end{definition}
 
To derive a sufficiently  small $ I_{\alpha}$-exhaustive set we employ techniques from global optimization theory.   
Specifically, we use a branch and prune algorithm which is derived from the classical global optimization technique of branch and bound  \cite{scholz2011deterministic,ratschek1988new,horst2013global}.    
Our branch and prune is designed so that it will output an  $I_\alpha$-exhaustive set, a result proved in Theorem~\ref{prop:BranchNPrune}.

The branch and prune algorithm begins with an initial finite set $\cS = \{ K_i : K_i \subset \R^3 \}$ for which $ \bigcup_{K \in \cS} K$ is  $I_\alpha$-exhaustive. 
The construction of this initial set is described in Section \ref{sec:InitialBounds}, specifically in Algorithms \ref{alg:InitialBoundsleq3} and \ref{alg:InitialBoundsgeq3}.  
We then alternate between branching and pruning the elements of $ \cS$. 
The branching subroutine divides an element $K \in \cS$ into two pieces $K_A$ and $K_B$ for which $ K = K_A \cup K_B$, and then replaces $K $ in the set $ \cS$ by the two smaller regions. 
The pruning algorithm uses a variety of techniques to derive sharper bounding functions on the region $K$. 
Furthermore, if we can prove that the preimage $ \kappa^{-1}(K) \subseteq \cX$ cannot contain any SOPS, then we remove the region $K$ from the set $ \cS$. 
The branch and prune algorithm terminates when the diameter of every region $K$ is less than some preset constant.

In contrast to the    prototypical optimization problem of bounding the minimum of an objective function, we are concerned with characterizing SOPS to Wright's equation. 
In particular, our pruning algorithm is designed to tighten the bounding functions associated with a region $K$, reduce the size of $K$, and to discard the region if we can prove that $ \kappa^{-1}(K)$ does not contain any SOPS.  
The algorithm takes as input an interval $I_\alpha$, a region $K$, and a pair of bounding functions $u_K, \ell_K$. 
As output the algorithm produces a region $ K' \subset K$ and a pair of bounding functions $ u_{K'} , \ell_{K'}$. 
The set $K$ is taken to be rectangular, that is $ K = I_q \times I_{\bar{q}} \times I_M$ where $I_q =[q_{min}, q_{max} ]$ and $I_{\bar{q}} = [ \bar{q}_{min}, \bar{q}_{max}]$ and $ I_M = [ M_{min}  , M_{max}]$. 
Additionally, this algorithm takes as input a computational parameter $ n_{Time} \in \N$ relating to how we store the bounding functions $u_K , \ell_K$ on the computer (see Appendix \ref{sec:Appendix}).

The six steps in the pruning algorithm (Algorithm \ref{alg:Prune}) are independent of one another  and can be implemented in any order. 
In Steps 1-4 we describe how to tighten the bounds on $K$, $u_K$ and $\ell_K$.  
Each step is constructed so that the output does not worsen the existing bounds. 
That is each step of the algorithm produces an output for which $ K' \subseteq K$ and the inequalities $ u_{K'} \leq u_K$ and $ \ell_{K'} \geq \ell_K$ hold. 
At the end of each step we update our input so that we use the improved bounds in the next step.  
That is, we define:
\begin{align}
K&\bydef K'  & u_{K} &\bydef u_{K'} & \ell_{K} &\bydef \ell_{K'}  
\label{eq:ImprovedNewBounds}
\end{align}
and subsequently modify $K'$, $u_{K'}$ and $ \ell_{K'}$ as described in each individual step. 
In Steps 5-6, we check conditions which would imply that the region $K$ cannot contain the $ \kappa$-image of SOPS to Wright's equation.  
If this is the case, the algorithm returns $ K = \emptyset$.  

\begin{algorithm}[Pruning Algorithm]
	\label{alg:Prune}
	This algorithm takes as input $ I_ \alpha = [ \alpha_{min} , \alpha_{max}]$,  
	$ K =  [q_{min}, q_{max} ] \times [ \bar{q}_{min}, \bar{q}_{max}] \times  [ M_{min}  , M_{max}] \subseteq \R^3$   and associated bounding functions $ \ell_K$ and $u_K$, as well as the computational parameter $ n_{Time} \in \N$.   
	The outputs consist of a region $ K' \subseteq \R^3$ and associated bounding functions $ \ell_{K'}$ and $ u_{K'}$. \newline
	
	\noindent
	Define $ I_q = [ q_{min}, q_{max}] $, $I_{\bar{q} } = [\bar{q}_{min} , \bar{q}_{max}]$ and $ I_{M} = [M_{min},M_{max}]$ as well as   $L_{min} \bydef q_{min} + \bar{q}_{min}$, $L_{max} \bydef q_{max} + \bar{q}_{max}$ and $I_L \bydef [ L_{min},L_{max}]$.  
	 
	
	\begin{enumerate}
		\item 
		We tighten the bounding functions associated with the region $K$ using 
		\begin{align}
		u_{K'} (t) &\bydef
		\begin{cases}
		\min \{ M_{max} , u_K(1) \}  & \mbox{ if } t=1\\
		\min\{0,u_K(t)\} & \mbox{ if } t \in [- \bar{q}_{min}, 0] \cup   [q_{max},L_{min}] \\
		u_K(t) & \mbox{ otherwise}
		\end{cases} 	\label{eq:Kref2a} \\
		\ell_{K'} (t) &\bydef
		\begin{cases}
		\max \{ M_{min} , \ell_K(1) \} & \mbox{ if } t=1\\
		\max\{0,\ell_K(t)\} & 
		\mbox{ if }			 t \in [-L_{min},\bar{q}_{max}] \cup  [0,q_{min}] \cup   [L_{max},L_{min}+ q_{min}] \\
		\ell_K(t) & \mbox{ otherwise.}
		\label{eq:Kref2b}
		\end{cases}
		\end{align}
		Lastly we update our bounds using Line \eqref{eq:ImprovedNewBounds}.
		
		\item  
		If $x$ satisfies Wright's equation we can use variation of parameters to refine the bounding functions.  
		For our computational parameter $ n_{Time} \in \N$, we base our calculation about a collection of points separated by a uniform distance of $ 1/n_{Time}$.  
		That is, define $ \Delta = 1/n_{Time}$ and $ I_\Delta \bydef [0,\Delta]$, and fix  $t_0 \in \{ k \cdot \Delta\}_{k \in \Z}$ and $ s \in I_\Delta$. 
		We may refine the values of 
		$ u_K(t_0+s)$, 
		$ u_K(t_0-s)$, 
		$ \ell_K(t_0+s)$, 
		$ \ell_K(t_0-s)$ as follows: 
		\begin{align} 
		u_{K''}(t_0+s) &\bydef u_K(t_0 ) 
		+ s \cdot \sup_{\alpha \in I_\alpha , r \in I_{\Delta} } 
		\sup_{ \ell_K \leq x \leq u_K}
			- \alpha  \left( e^{  x ( t_0 -1 + r) }-1 \right) \label{eq:RiemmanUpper} \\ 
		u_{K''}(t_0 - s) &\bydef u_K(t_0 )  - s \cdot \inf_{\alpha \in I_\alpha , r \in I_{\Delta} } 
		\inf_{ \ell_K \leq x \leq u_K}
		 - \alpha  	\left( e^{  x ( t_0 -1 - r) }-1 \right) \nonumber \\ 
		 \ell_{K''}(t_0 + s) &\bydef \ell_K(t_0 )  + s \cdot
		  \inf_{\alpha \in I_\alpha , r \in I_{\Delta} }
		  \inf_{ \ell_K \leq x \leq u_K}
		 - \alpha  	\left( e^{  x ( t_0 -1 + r) }-1 \right) \nonumber  \\
		 \ell_{K''}(t_0 - s) &\bydef \ell_K(t_0 )  - s \cdot 
		 \sup_{\alpha \in I_\alpha , r \in I_{\Delta} }
		 \sup_{ \ell_K \leq x \leq u_K}
		 - \alpha  	\left( e^{  x ( t_0 -1 - r) }-1 \right) \nonumber 
		\end{align}
		and
		\begin{align*} 
		u_{K'}(t_0+s)  &\bydef \min \left\{ u_K(t_0+s) , u_{K''}(t_0+s)    \right\} \\
		u_{K'}(t_0 - s)  &\bydef \min \left\{ u_K(t_0-s) , u_{K''}(t_0-s)    \right\} \\
		\ell_{K'}(t_0+s)  &\bydef \max \left\{ \ell_K(t_0+s) , \ell_{K''}(t_0+s)    \right\} \\ 
		\ell_{K'}(t_0-s)  &\bydef \max \left\{ \ell_K(t_0-s) , \ell_{K''}(t_0-s)    \right\} .
		\end{align*}
		
		Appendix \ref{sec:Appendix} explains in further detail the computational aspects of this step.
		Lastly we update our bounds using Line \eqref{eq:ImprovedNewBounds}.
		
		\item 	
		In this step we refine our bounds on $I_q$ and $I_{M}$ using $ u_K$ and $ \ell_K$.  
		At $ t = q(x)$ the function $x(t)$ changes sign from positive to negative. 
		We sharpen the bounds on $I_q$ by defining:
		\begin{align}
		q'_{min} &\bydef \inf \{ t \in I_q : \ell_K(t) \leq 0 \} \nonumber \\
		q'_{max} &\bydef \sup \{ t \in I_q : u_K(t) \geq 0 \} 	\nonumber \\
		I_{q'}  &\bydef [ q'_{min} , q'_{max} ]. \label{eq:QFind} 
		\end{align}
		Additionally we make the following refinement:
		\[
		I_{M'} \bydef [M_{min} , M_{max}] \cap [\ell_{K}(1) , u_{K}(1)].
		\]
		Lastly we define $K' \bydef I_{q'} \times I_{\bar{q}} \times I_{M'}$ and update our bounds using Line \eqref{eq:ImprovedNewBounds}. 
		
		\item 
		If $ x \in \cX$ is a SOPS with period $L \in I_L$, then $ x(t) = x(t +L)$. 
		Using this relation, we make the following refinement: 
		\begin{align}
		\ell_{K'}(t) &\bydef  \max \left\{ \ell_K(t) , \min_{L' \in I_L}  \ell_K( t + L') \right\}  \label{eq:Alg1Period1} \\	
		u_{K'}(t)    &\bydef  \min \left\{ u_K(t) , \max_{L' \in I_L}  u_K( t + L') \right\} \label{eq:Alg1Period2}.
		\end{align}	
		Lastly we update our bounds using Line \eqref{eq:ImprovedNewBounds} as appropriate.
		
		\item 
		If there is some point $t\in \R$ for which $ \ell_{K} (t) > u_{K}(t)$ then \emph{RETURN}  $K'  \bydef \emptyset$. 
		
		\item
		If $\min_{ t \in I_q} \ell_K(t + 1)  > - \log \tfrac{\alpha_{min}}{\pi/2}$, then \emph{RETURN} $K'  \bydef \emptyset$.

	\end{enumerate}
	
\end{algorithm}

\begin{proposition}
	Let $ I_\alpha = [ \alpha_{min} , \alpha_{max}]$, $\alpha_{min} \geq  \pp$, 	$$ K =  [q_{min}, q_{max} ] \times [ \bar{q}_{min}, \bar{q}_{max}] \times  [ M_{min}  , M_{max}] \subseteq \R^3$$   and $u_K , \ell_K$ be input for Algorithm \ref{alg:Prune} with any computational parameter $n_{Time} \in \N$. 
	Suppose that $ x \in \cX$ is a SOPS at parameter $\alpha \in I_\alpha$, 
	and   let  $\{K',u_{K'},\ell_{K'}   \} $ be the result of Algorithm \ref{alg:Prune}. 
	If $ \kappa(x) \in K$, then $ \kappa(x) \in K'$ and $ \ell_{K'} \leq x \leq u_{K'}$. 
	\label{prop:Prune}
\end{proposition}

\begin{proof}

We prove that Proposition \ref{prop:Prune} holds for each step of the algorithm individually. 
Given an interval $ I_\alpha \subset \R$, let $ x \in \cX$ be a SOPS at parameter $\alpha \in I_\alpha$.

\begin{enumerate}
	
	\item  
		Recall $ \kappa(x) = \{ q(x) , \bar{q} (x) , \max(x) \}$. 
		Since $\max_{t \in \R} x = x(1)$ then the  refinements in  \eqref{eq:Kref2a} and \eqref{eq:Kref2b} for the case in which $ t = 1$ are appropriate.  
		For the other two refinements in each equation note that by definition  a function $ x \in \cX$ is non-negative on the interval $ [0 ,q(x)]$ and non-positive on the interval $ [ q(x), q(x) + \bar{q}(x) ]$. 
		Hence, $x$ is non-negative on $ [0 , q_{min}]$ and non-positive on $ [q_{min}, L_{min}]$. 
	If $x$ is a SOPS then it has period $L=q(x) + \bar{q}(x) $ and we may further conclude that it is non-negative on the intervals $ [-L ,\bar{q}(x)]$ and $ [L ,L+q(x)]$, and non-positive on the interval $ [ -\bar{q}(x), 0 ]$.  
	Hence, $x$ is non-negative on $ [-L_{min} , \bar{q}_{max}]$ and $ [L_{max} , L_{min} +q_{min}]$, and non-positive on $ [ - \bar{q}_{min} ,0]$. 
 The refinements in \eqref{eq:Kref2a} and \eqref{eq:Kref2b}  reflect these restrictions.

	\item 	
		
		To estimate an upper bound on $ x(t_0 +s)$, we apply variation of parameters to Wright's equation, obtaining 
		\begin{equation}
			x(t_0 + s) = x(t_0) + \int_{t_0}^{t_0 + s} - \alpha \left( e^{ x(r-1)} -1 \right) dr.
		\end{equation}
		Taking the Riemann upper sum of this integral with step size $ s$, we deduce that $x(t_0 + s)$ is bounded above by the RHS of  \eqref{eq:RiemmanUpper}.  
		As $ x(t_0 +s) \leq u_K(t_0+s)$ it follows that $x(t_0 +s ) \leq \min \{ u_{K}(t_0+s) , u_{K''}(t_0+s)  \}$.  
		The proofs for the refinements of $ u_K(t_0-s), \ell_K(t_0 +s), \ell_K(t_0-s)$ follow with parity.

	\item  
	Let $ x \in \cX$ be such that $ \kappa(x) \in K$. 
	From our definitions of $ q'_{min}$ and $q'_{max}$ it follows that 
	\begin{align*}
		x(t) \geq \ell_K(t) > 0, & \quad \mbox{for all } t \in  \left( q_{min} , q_{min}' \right), \\
		x(t) \leq u_K(t) < 0, & \quad \mbox{for all } t \in  \left( q_{max}' , q_{max}\right).
	\end{align*} 
	Hence it follows that $x(t) \neq 0$ for $ t \in ( q_{min} , q_{min}') \cup  ( q_{max}' , q_{max})$.    
	Since $ q(x) \in I_{q}$,   it must  follow that $ q(x) \in [ q_{min}', q_{max}']$, thus justifying the refinement in \eqref{eq:QFind}. 
	Regarding the refinement of $ I_M$, as $ \ell_K(1) \leq x(1) \leq u_K(1)$ it clearly follows that $ [ \kappa(x) ]_3 = x(1) \in I_{M'} $.

	\item  
		If $x$ is periodic with period $L$, then $ x (t) = x(t + L)$. Since $ L \in I_L$ then we may derive upper/lower bounds on  $x(t + L) $ as follows: 
	\[
 \min_{L'\in I_L} \ell_K(t+L') \leq \min_{L' \in I_L} x(t+L') \leq 	x(t+L) \leq \max_{L' \in I_L} x(t+L') \leq \max_{L'\in I_L} u_K(t+L').
	\]
	Hence it follows that $ \min_{L'\in I_L} \ell_K(t+L') \leq \ell_K(t)$ and $u_K(t) \leq \max_{L'\in I_L} u_K(t+L')$, thus justifying our refinements  in~\eqref{eq:Alg1Period1} and \eqref{eq:Alg1Period2}.

	\item 
		If $ \ell_K(t) >  u_K(t)$,  then it is impossible for any $x \in \cX$ to satisfy $ \ell_K(t) \leq x(t) \leq u_K(t)$. 
		Since $u_K$, and $\ell_K$ are bounding functions associated with $K$, this contradiction leads us to conclude that there cannot exist any SOPS $x \in \cX$ for which $ \kappa(x) \in K$.

	\item
	 	By the results in \cite{walther1978theorem}, if $x$ is a SOPS to Wright's equation and $ \alpha \geq \pp$, then 
\begin{equation}
	 	\min x \leq - \log   \tfrac{ \alpha}{\pi /2}.  \label{eq:Walther}
\end{equation}
	 	If $x \in \cX$ is a SOPS then $ \min_{t \in \R} x(t) = x(q +1)$, whereby $ \min_{t \in \R} x(t) > \min_{ t \in I_q} \ell_k(t + 1)  $. 
	 	Hence, if $ \min_{ t \in I_q} \ell_k(t + 1)  >  - \log   \tfrac{\alpha_{min}}{\pi/2 } $ then  \eqref{eq:Walther} is violated, and so there cannot exist any SOPS $x \in \cX$ for which $ \kappa (x) \in K$.  \qedhere

\end{enumerate}

\end{proof}

\section{Initial Bounds on SOPS to Wright's Equation}
\label{sec:InitialBounds}

In order to apply the branch and prune algorithm, we must first construct an initial $I_\alpha$-exhaustive set. 
Due to the sustained interest in Wright's equation, there are considerable \emph{a priori} estimates we can employ to describe slowly oscillating solutions \cite{jones1962nonlinear,wright1955non,nussbaum1982asymptotic}. 
Since considerably sharper estimates are obtained under the assumption $ \bar{q} \geq 3$, we will construct two regions $K_1$ and $K_2$ 
corresponding to SOPS $ x \in \cX$ for which $ \bar{q} (x) \leq 3$ and $ \bar{q} (x) \geq 3$ respectively. 
Taken together $K_1\cup K_2$ will form an $I_\alpha$-exhaustive set, which we prove in Corollary \ref{prop:ConstructInitialBounds}.

While sharper estimates are available for additional sub-cases \cite{nussbaum1982asymptotic},  we present a collection of these estimates we have found sufficient for our purposes. 
Note that these lemmas are not a verbatim reproduction. We have translated results applicable to the quadratic form of Wright's equation given in \eqref{eq:Wright} so that they apply to the exponential form of Wright's equation given in \eqref{eq:FeedBack}.

\begin{lemma}[See \cite{wright1955non}] 
 Let $ x \in \cX$ be a solution to Wright's equation at parameter $ \alpha >0$. 
 Then 
 \[
 - \alpha ( e^{\alpha} -1) \leq x(t) \leq \alpha
 \]
 for all $ t >0 $.
 	\label{prop:GlobalExtrema}
\end{lemma}

\begin{lemma}[See {\cite[Theorem 3.1]{jones1962nonlinear}}] 
	Let $ \alpha > e^{-1}$ and suppose that $ x \in \cX$ and is a solution to Wright's equation. 
	We construct a sequence of functions $  p_i: (-\infty , 1] \to \R $ for $i = 1, 2, \cdots$ 	by setting $p_1(t) = \alpha t $ and recursively defining:
	\[
	p_{i+1} (t)\bydef - \alpha \int_0^t \left( e^{p_i(s-1)} -1 \right) ds.
	\]
	For example $ p_2( t) = \alpha t + e^{-\alpha} -e^{\alpha(t-1)}$. 
	Then $ x(t) > p_i(t)$ for $ t < 0 $, and $ x(t)  < p_i( t )$ for $t \in (0,1]$. Furthermore $ x(t) < p_i(1) $ for all $t \geq 0$. 	 
	Additionally $ |p_i(t)| $ is increasing in $\alpha$. 
	\label{prop:JonesMax}
\end{lemma}

\begin{lemma}[See  {\cite[Theorem 3.4]{jones1962nonlinear}}]
	\label{prop:MinusOne}
	Let $ x \in \cX$ and suppose that $\bar{q} \geq 3$ and that $ \alpha \geq \pp$. 
	Define $a_1(\alpha) = - ( \alpha -1 )$ and the recursive relation $a_{i+1}(\alpha) =  \alpha (  e^{a_i(\alpha)} -1 ) $. Then  $x(t) < -t \cdot a_i(\alpha)  $ for $t \in [-1,0)$  and $ i \in \N $.
		
\end{lemma}

\begin{lemma}
Suppose that $ x \in \cX $ is a SOPS to Wright's Equation. 
If $ \alpha \geq \pp$ then 
\[
\begin{array}{ccccc}
1+ \frac{1}{\alpha} \left( \frac{\alpha + e^{-\alpha}-1}{\exp\{\alpha + e^{-\alpha}-1\}-1} \right) &<& q  &<& 2 + \frac{1}{\alpha} \\
1 + \frac{1}{\alpha} &<&\bar{q}&<& \max \{3, 2 + | \frac{e^\alpha-1}{e^{a_i(\alpha)}-1}| \}
\end{array}
\]
where $a_i(\alpha)   $ is taken as in Lemma \ref{prop:MinusOne}. 
Additionally, if $ \alpha \geq 2$ then $ q <2$. 
\label{prop:Qbound}
\end{lemma}
\begin{proof}
	All but the upper bound on $q$ follows from Theorem 3.5 in \cite{jones1962nonlinear}. 
	To prove the upper bound, assume that $ q \geq 2$, and consider the quadratic version of Wright's equation given in \eqref{eq:Wright}.  It follows that $ y(t) \geq y(2) $ for all $t \in [1,2]$,  
	and thereby $ y'(t) \geq -\alpha y(2) [1+0]$ for all $t \in [2,3]$. 
	From this we obtain $ q < 2 + \tfrac{1}{\alpha}$.
\end{proof}
 
Step 2 constructs iterative bounds analogous to Lemmas \ref{prop:JonesMax} and \ref{prop:MinusOne}. 
When $\bar{q}(x) \geq 3$,  Step 3  obtains bounds on $I_M$ and  $I_q$ which are tighter than the bounds given in Lemma \ref{prop:GlobalExtrema} and Lemma \ref{prop:Qbound}.
Additionally, the branching procedure further reduces the size of $ I_q$, $I_{\bar{q}}$, and $I_{M}$. 
Below in Algorithm \ref{alg:InitialBoundsleq3} we construct the initial bounds for a region $ K \subseteq \R^3$ containing the $\kappa$-image of SOPS $ x \in \cX$ for which $ \bar{q}(x) \leq 3$.

\begin{algorithm}
	The input we take is an interval $ I_\alpha = [ \alpha_{min} , \alpha_{max}]$ and computational parameters $i_0 , n_{Time} \in \N$. 
	The output is a rectangle $K =I_{q} \times I_{\bar{q}} \times I_{M} \subseteq \R^3$  and bounding functions $u_{K}$, $\ell_{K}$.
	\begin{enumerate}
		
		\item 
		Make the following definitions:
		\begin{align*}
		q_{min} & \bydef 1+  \inf_{\alpha \in I_\alpha} \tfrac{1}{\alpha} \left( \tfrac{\alpha + e^{-\alpha}-1}{\exp\{\alpha + e^{-\alpha}-1\}-1} \right) \\
		M_{min} & \bydef \inf_{\alpha \in I_\alpha } \log \left( 1 + \alpha^{-1} \log  \tfrac{\alpha }{\pi/2}\right) \\
		I_q 			&\bydef 
		\begin{cases}
		\left[ q_{min} ,2  \right] &\mbox{ if } \alpha_{min} \geq 2  \\
		\left[ q_{min} , 2 + \tfrac{1}{\alpha_{min}} \right] & \mbox{otherwise}
		\end{cases}
		\\
		I_{\bar{q} } 		&\bydef
									\left[ 1 + \tfrac{1}{\alpha_{max}} , 3\right] \\
		I_M 				&\bydef 
									 \left[ M_{min},p_{i_0}(1) \right].
		\end{align*}
		
		\item For $p_i$ given as in Propositions  \ref{prop:JonesMax}, define bounding functions 
		\begin{align*}
		\ell_{K}(t) \bydef&
		\begin{cases}
		0						& \mbox{ if } t = 0  \\
		- \alpha_{max} (e^{\alpha_{max}} - 1) & \mbox{ otherwise} 
		\end{cases}
		&
		u_{K}(t) \bydef&
		\begin{cases}
		0& \mbox{ if } t = 0  \\
		p_{i_0}(1) & \mbox{ otherwise.} 
		\end{cases}
		\end{align*}
		These bounding functions are stored on the computer with time resolution $n_{Time}$ as described in Appendix \ref{sec:Appendix}. 
		
	\end{enumerate}
	
	\label{alg:InitialBoundsleq3}
\end{algorithm}

\begin{proposition}
		\label{prop:InitialShort}
	Fix an interval $ I_\alpha = [ \alpha_{min} , \alpha_{max}]$ such that $\alpha_{min} \geq \pp$. Let $K , u_{K}$, $\ell_{K}$ denote the output of Algorithm \ref{alg:InitialBoundsleq3}. 
	If $ x \in \cX$ is a SOPS to Wright's equation at parameter $ \alpha \in I_\alpha$ and $ \bar{q} (x) \leq 3$ then $ \kappa ( x) \in K$ and $ \ell_K \leq x \leq u_K$. 
\end{proposition}

\begin{proof}
	We treat the two steps in order. 
	\begin{enumerate}
		\item
		If $ x \in \cX$ is a SOPS to Wright's equation and $ \bar{q} (x) \leq 3$, then by Lemma \ref{prop:Qbound} it follows that $ q(x) \in I_q$ and $ \bar{q}(x) \in I_{\bar{q}}$. 
		 By Proposition  \ref{prop:JonesMax} it follows that  $ x(1) \leq p_{i_0}(1)$.  
		If $x$ is a SOPS to Wright's equation with $ \alpha \geq \pi/2$ then $ \min x \leq - \log   \tfrac{2 \alpha}{\pi } $ \cite{walther1978theorem}.  
		If  $\max x <  \log \left( 1 + \alpha^{-1} \log  \tfrac{  \alpha }{\pi/2} \right)   $  
		then by integrating Wright's equation forward from $t=q$ to $t=q+1$  it follows that $ x(q+1) = \min x \leq - \log   \tfrac{  \alpha}{\pi /2} $, a contradiction.
		Hence we may assume that $x(1)  \geq   \log \left( 1 + \alpha^{-1} \log  \tfrac{ \alpha }{\pi/2} \right) \geq M_{min} $.

		\item
		Since $x \in \cX$ then $ x(0) = 0$, and by Lemma \ref{prop:GlobalExtrema} and Proposition  \ref{prop:JonesMax} it follows that  $ - \alpha ( e^\alpha -1) \leq x \leq p_{i_0}(1)$ for  any SOPS $ x \in \cX$. 
		Hence $\ell_K$ and $u_K$ are bounding functions for $ K = I_{q} \times I_{\bar{q}} \times I_{M}$.  \qedhere
	\end{enumerate}
\end{proof}

To construct the initial bounds for the case $ \bar{q}(x) \geq 3$,  we make greater use of \emph{a priori} bounds.  
Unfortunately the bounds on $I_{\bar{q}}$ given in Lemma \ref{prop:Qbound} are not sharp, that is the width of this estimate of $I_{\bar{q}}$ is greater than $ e^\alpha -2$. Using this estimate would be computational difficult.  
In \cite{nussbaum1982asymptotic} Nussbaum estimates the value of $ \bar{q}$ up to $\mathcal{O}(\tfrac{1}{\alpha})$ in the case of $ \bar{q}(x) \geq 3$ and $ \alpha \geq 3.8$.  
We derive a similar estimate which only assumes $\bar{q}(x) \geq 2$ and $\alpha > 0$. 
This estimate is better suited for numerical applications, and only needs bounds $ \ell(t) \leq x(t) \leq u(t)$ that are defined over the time domain $ t \in [-1,4]$.

	\begin{lemma}
		Fix some $ \alpha >0$ and suppose that $x \in \cX$ is a SOPS to Wright's equation, and 
		let $\ell,u:\R \to \R$ be functions for which $\ell(t) \leq x(t) \leq u(t)$.  
		Let $ I_q \subset \R$ be an interval for which $ q(x) \in I_q$ and suppose that $ \bar{q}(x) \geq 2$. 
		Define the following integral bounds: 
		\begin{align}
				U^{+} &\bydef  		\sup_{q \in I_q } 	\int_{q-1}^q 	\max\left \{ e^{u(t)} -1,0\right\} dt   
			&	U^{-}_{1} &\bydef  	\sup_{q \in I_q }   \int_{q}^{q+1} 	- \min\left \{ e^{\ell(t)} -1 ,0\right\} dt  \label{eq:Udef}\\
				L^{+} &\bydef  		\inf_{q \in I_q }  	\int_{q-1}^q 	\max\left \{ e^{\ell(t)} -1,0\right\} dt  
			&	L^{-}_{1} &\bydef  	\inf_{q \in I_q }   \int_{q}^{q+1} 	- \min\left \{ e^{u(t)} -1 ,0\right\} dt  \label{eq:Ldef}
		\end{align}
		 and define 
		\(
		m \bydef \min_{   t \in I_q } \ell(t+1)
		\).  
		Then $\bar{q}$ is bounded by the inequalities 
	\begin{equation}
2 + \frac{L^{+} - U^{-}_{1}}{|e^{m} -1|} \leq \bar{q} \leq 2 + \frac{U^{+} - L^{-}_{1}}{|e^{u(-1)} -1|} .
\label{eq:PeriodBound}
	\end{equation}
		\label{prop:PeriodBound}
		\end{lemma}
		 The proof is  delayed until the end of this section.
		 The computational details of how we evaluate the integrals in  \eqref{eq:Udef} and \eqref{eq:Ldef} are discussed in Appendix \ref{sec:Appendix}. 
		 Below in Algorithm \ref{alg:InitialBoundsgeq3} we construct the initial bounds for a region $ K \subseteq \R^3$ containing the $\kappa$-image of SOPS $ x \in \cX$ for which $ \bar{q}(x) \geq 3$.

\begin{algorithm}
The input  is an interval $ I_\alpha= [ \alpha_{min} , \alpha_{max}]$ and computational parameters $i_0, j_0, \\ n_{Time}, \ N_{period} \in \N$. 
The output is a rectangle $K =I_{q} \times I_{\bar{q}} \times I_{M}$  and bounding functions $u_{K}$, $\ell_{K}$.
\begin{enumerate}
	
	\item 
	Make the following definitions for $K =I_{q} \times I_{\bar{q}} \times I_{M}$:
	\begin{align*}
		q_{min} & \bydef 1+  \inf_{\alpha \in I_\alpha} \tfrac{1}{\alpha} \left( \tfrac{\alpha + e^{-\alpha}-1}{\exp\{\alpha + e^{-\alpha}-1\}-1} \right) \\ 
		I_q 			&\bydef
		\left[q_{min} ,
		2 + \tfrac{1}{\alpha_{min}} 
		\right] \\
	I_{\bar{q} } 	&\bydef 
	\left[ 3 , \sup_{\alpha \in I_\alpha} 2 + \left| \frac{e^\alpha-1}{e^{a_{j_{0}}(\alpha)}-1} \right| \right] \\
	I_M 				&\bydef \left[0, p_{i_0}(1) \right]
	\end{align*}
	where $a_i(\alpha)  $ is taken as in Lemma \ref{prop:MinusOne}. 
	
	\item For $p_i$ and $a_j$ given as in Propositions  \ref{prop:JonesMax} and \ref{prop:MinusOne} respectively,  define bounding functions $ \ell_{K}$ and $ u_{K} $ 
		\begin{align*}
			\ell_{K}(t) \bydef&
			\begin{cases}
			0& \mbox{ if } t = 0  \\
			p_{i_{0}}(t)  & \mbox{ if } t < 0   \\
			\inf_{\alpha \in I_\alpha}- \alpha (e^\alpha - 1) & \mbox{ otherwise} 
			\end{cases}
			&
			u_{K}(t) \bydef&
			\begin{cases}
			0& \mbox{ if } t = 0  \\
			 \sup_{\alpha \in I_\alpha }  -t \cdot a_{j_{0}}(\alpha)     & \mbox{ if } t \in [-1,0)  \\
			p_{i_0}(1) & \mbox{ otherwise.} 
			\end{cases}
		\end{align*}
		These bounding functions are stored on the computer with time resolution $n_{Time}$ as described in Appendix \ref{sec:Appendix}.

	\item 
	Refine $u_K$ and $\ell_K$ according to Step 1 of Algorithm \ref{alg:Prune}. 
	 For  $N_{period}$  iterations, refine $u_K$ and $\ell_K$ according to Step 2 of Algorithm \ref{alg:Prune} for values $ t_0 \in [-4,4]$. Then define $I_q$ and $I_M$ according to Step 3 of Algorithm \ref{alg:Prune}.

	\item 
	For values of $m,L^{+},L^{-}_{1},U^{+},U^{-}_{1}$  given as in Proposition \ref{prop:PeriodBound}, define:
\begin{align*}
\bar{q}_{min} \bydef& 2 + \frac{L^{+} - U^{-}_{1}}{|e^{m} -1|} ,
&
\bar{q}_{max} \bydef &  
2 + \frac{U^{+} - L^{-}_{1}}{|e^{u_K(-1)} -1|}.
\end{align*}
	If $\bar{q}_{max} < 3$ then define $K = \emptyset $. 
	Otherwise define $ I_{\bar{q}} = [ \bar{q}_{min} ,\bar{q}_{max} ]$ and $ K = I_q \times I_{\bar{q}} \times I_M$.

\end{enumerate}

\label{alg:InitialBoundsgeq3}

\end{algorithm}

\begin{remark}
	In practice we select $i_0 =2$ and $ j_0 =20$ in Step 2, which have proved sufficient for our purposes. 
	In \cite{jones1962nonlinear} the expressions for $ p_i$  are given in closed form for $ i = 1,2,3,4$, each function being increasingly complex. 
	The sequence $ a_j(\alpha)$ is convergent, and we use $j_{0}=20$ because we have found negligible improvements when   using a larger index. 
\end{remark}

\begin{proposition}
	\label{prop:InitialLong}
	Fix an interval $ I_\alpha = [ \alpha_{min} , \alpha_{max}]$ such that $ \alpha_{min} \geq \pp$, and fix computational parameters $i_0, j_0,n_{Time} , N_{period}  \in \N$. Let $\{K , u_{K}$, $\ell_{K}\}$ denote the output of Algorithm \ref{alg:InitialBoundsgeq3}. 
	If $ x \in \cX$ is a SOPS to Wright's equation and $ \bar{q} (x) \geq 3$ then $ \kappa ( x) \in K$ and $ \ell_K \leq x \leq u_K$. 
\end{proposition}

\begin{proof}
	Let $x$ be as described above. We describe the effect of each step of the algorithm in turn. 
	
	 \begin{enumerate}
		\item  
		For $ I_q, I_{\bar{q}}$ and $ I_M$ defined in Step 1, it follows from Lemma  \ref{prop:Qbound} that $ q(x) \in I_q$ and $ \bar{q}(x) \in I_{\bar{q}}$, and it follows from Lemma \ref{prop:GlobalExtrema} and Lemma \ref{prop:JonesMax} that $ x(1) \in I_M$. 

		\item 
		Since $ x \in \cX$ then $x (0) = 0$.  
		By Lemma  \ref{prop:GlobalExtrema}  then any SOPS $ x \in \cX$ satisfies the inequality $ - \alpha ( e^\alpha -1) \leq x(t) \leq p_{i_0}(1)$.  
		The definition of the $ \ell_K$ bound for $ t<0$ follows from Lemma \ref{prop:JonesMax}, and the definition of the $u_K$  bound  for $ t \in [-1,0)$  follows from  Lemma \ref{prop:MinusOne}.
		
		\item 
		The results of Steps 1 and 2 produce a region $K$ with bounding functions $u_K, \ell_K$  for which $ \kappa(x) \in K$ whenever there is a SOPS $x \in \cX$ satisfying $ \bar{q}(x) \geq 3$. 
		By Proposition \ref{prop:Prune}, implementing Steps 2 and 3 of Algorithm \ref{alg:Prune} preserves this property. 
		
		\item 
		Since $ \bar{q}(x) \geq 3 > 2$ then by Lemma \ref{prop:PeriodBound} it follows that $ \bar{q}_{min} \leq \bar{q}(x) \leq \bar{q}_{max}$.  
		If $ \bar{q}_{max} < 3$, this contradicts our initial assumption that $ \bar{q}(x) \geq 3$, whereby there are no SOPS $x \in \cX$ to Wright's equation at any parameter $\alpha\in I_\alpha$ for which $\bar{q}(x) \geq 3$.   
		Otherwise for our definition of $ K  = I_q \times I_{\bar{q}} \times I_M$ it follows that $ \kappa(x) \in K$ whenever $ \bar{q}(x) \geq 3$.  \qedhere

	 \end{enumerate}
\end{proof}

 We present an application of this theorem. 
 \begin{proposition}
 	If $ x \in \cX$ is a SOPS to~\eqref{eq:Wright} and $ \alpha \in [\pp , 2.07]$ then $ \bar{q}(x) < 3$. 
 \end{proposition}
\begin{proof}
	First we constructed subintervals $I_{\alpha}$ of $ [ 1.57, 2.07]$ of width $ 0.1$, and for each subinterval $I_\alpha$  we ran Algorithm \ref{alg:InitialBoundsgeq3} with computational parameters $i_{0} =2$, $j_{0} =20$, $ n_{time} = 128$, and $ N_{period} = 10$ (see \cite{webpage_jones_conjecture} for associated MATLAB code). 
	In each case the algorithm returned $K = \emptyset$. 
\end{proof}

\begin{corollary}
	Fix an interval $I_\alpha = [ \alpha_{min} , \alpha_{max}]$ such that $ \alpha_{min}  > \pp$, and fix computational parameters $i_0, j_0, N_{period} \in \N$.  
	Let $ \{K_1 , u_{K_1}$, $\ell_{K_1}\} $ denote the output of Algorithm \ref{alg:InitialBoundsleq3} and let $\{K_2 , u_{K_2}$, $\ell_{K_2}\}$ the output of Algorithm \ref{alg:InitialBoundsgeq3}. 
	Then $  K_1 \cup K_2 $ is $I_\alpha$-exhaustive. 
	\label{prop:ConstructInitialBounds}
\end{corollary}

\begin{proof}
	Suppose that $x \in \cX$ is SOPS to Wrights equation. 
	If $ \bar{q}(x) \leq 3$, then by Proposition \ref{prop:InitialShort} it follows that $ \kappa(x) \in K_1$. 
	If $ \bar{q}(x) \geq 3$, then by Proposition \ref{prop:InitialLong} it follows that $ \kappa(x) \in K_2$. 
	Hence the set $  K_1 \cup K_2  $ is $I_\alpha$-exhaustive. 
\end{proof}

\begin{proof}[Proof of Lemma \ref{prop:PeriodBound}]
	
	Let $ p $ denote the period of a SOPS $x \in \cX$. 
	By assumption $ x(p)= x(q) =0$,  so by the fundamental theorem of calculus we have that for any SOPS $x$, 
\[
0 = x(p)- x(q)	  = \int_q^p x'(t) dt  =  \int_q^p - \alpha ( e^{x(t-1)} -1) dt = \int_{q-1}^{p-1} ( e^{x(t)} -1) dt	.
\]
	
	Recall that any SOPS $ x(t)$ is positive for $ t \in ( 0 , q)$ and negative for $ t \in ( q , p)$. 
	Hence the integrand above is positive on $ (q-1,q)$ and negative on $ ( q , p-1)$, thus producing the following estimate:
	\begin{equation}
	\int_{q-1}^q | e^{x(t)} -1| dt = \int_q^{p-1} |e^{x(t) }-1| dt  .
	\label{eq:PeriodOne}
	\end{equation}
	For $ t \in (q-1,q)$ the function $x(t)$ is positive, whereby $|e^{x(t)} -1| =  \max\{e^{x(t)} -1,0\}$. 
	For the definitions of $ L^{+}$ and $U^{+}$ given in \eqref{eq:Udef} and \eqref{eq:Ldef}, it follows that  $ L^{+}$ and $U^{+}$  bound the LHS of  \eqref{eq:PeriodOne} as described below:
	\[
	L^{+} 
	\leq \int_{q-1}^q \max\{ e^{\ell(t)} -1,0\}  dt 
	\leq \int_{q-1}^q |e^{x(t)} -1|  dt 
	\leq \int_{q-1}^q \max\{ e^{u(t)} -1,0\}  dt 
	\leq U^{+}.
	\]
	We estimate the RHS of \eqref{eq:PeriodOne} using the two sums below: 
	\[
	L^{-}_{1} + L^{-}_{2} 
	\leq \int_q^{p-1} | e^{x(t)} -1| dt 
	\leq U^{-}_{1} + U^{-}_{2}
	\]
	where the constants $L^{-}_{1}, L^{-}_{2} , U^{-}_{1} , U^{-}_{2}$ are appropriately defined so that
	\begin{eqnarray}
	L^{-}_{1} \leq \int_{q}^{q+1} | e^{x(t)} -1 | dt 	 \leq U^{-}_{1} \label{eq:LUn1} \\		
	L^{-}_{2} \leq \int_{q+1}^{p-1} | e^{x(t)} -1 | dt \leq U^{-}_{2} . \label{eq:LUn2}
	\end{eqnarray}
	For $ t \in (q,q+1)$ the function $x(t)$ is negative, whereby $|e^{x(t)} -1| =  - \min\{e^{x(t)} -1,0\}$. 
	It follows from  the definitions of $ L^{-}_{1}$ and $ U^{-}_{1}$ given in  \eqref{eq:Udef} and \eqref{eq:Ldef}  that  \eqref{eq:LUn1} is satisfied. 
	To define $ L^{-}_{2}$ and $U^{-}_{2}$ note that for the time period $ t \in [ q+1,p-1]$ we have that $ x'(t) >0$, whereby
	\begin{eqnarray*}
	x(t)  &\geq& x(q + 1) \\
	x(t)  &\leq& x(p-1)=x(-1) \leq u(-1).
	\end{eqnarray*} 
	By definition $m \leq x( q+1)$, and as  $ p-q = \bar{q}$ we can then define 
	\begin{align*}	
		U^{-}_{2} \bydef&\  \int_{q+1}^{p-1} | e^{m} -1 | dt  		&L^{-}_{2} \bydef&\  \int_{q+1}^{p-1} | e^{u(-1)} -1 | dt \\
		=&\ (\bar{q} -2) |e^m -1| 							&=&\ (\bar{q} -2) |e^{u(-1)} -1| .
	\end{align*}
	Using these definitions, \eqref{eq:LUn2} is  satisfied. 
	From \eqref{eq:PeriodOne},  we get the following upper and lower bounds on $ \bar{q}$, from which \eqref{eq:PeriodBound} follows. 
	\begin{align*}
	L^{-}_{1} + L^{-}_{2} \leq & \ U^{+} 									&U^{-}_{1} + U^{-}_{2} \geq &\ L^{+} \\
	(\bar{q} -2) |e^{u(-1)}-1| \leq&\ U^{+} - L^{-}_{1} 					&(\bar{q} -2) |e^{m}-1| \geq&\ L^{+} - U^{-}_{1} \\
	\bar{q} \leq&\ 2 + \frac{U^{+} - L^{-}_{1}}{|e^{u(-1)} -1|}			&\bar{q} \geq&\ 2 + \frac{L^{+} - U^{-}_{1}}{|e^{m} -1|}. \qedhere
	\end{align*}

\end{proof}


\section{Bounding the Floquet Multipliers. }
\label{sec:Floquet}

In this section we describe how to estimate the Floquet multipliers of SOPS contained within the bounds derived in  Sections \ref{sec:BrandAndPrune} and \ref{sec:InitialBounds}.
This method follows the approach of \cite{xie1991thesis} with modifications to take advantage of numerical computations.  
To review this method, we first define a hyperplane in $C[-1,0]$ as
\[
H \bydef \{ \varphi \in C [-1,0] : \varphi(0) = 0 \} .
\]
For a function $y$ we define $ y_0 \in C[-1,0]$ to be the cut-off function of $ y$ on $ [-1,0]$, and for a constant $ L \in \R$ we define $ y_L \bydef [y(t+ L)]_0 $.  
Locally, one can construct a smooth Poincar\'e map $ \Phi : H \to H$ via the solution operator. 
If $x \in \mathcal{X}$ is a SOPS, then $ x_0$ is a fixed point of $ \Phi$, and
the Floquet multipliers of $x$ are the eigenvalues of $ D_\varphi \Phi(x_0)$. 
Of course $x_0$ is a trivial eigenfunction with associated eigenvalue $ \lambda = 1$. 
The nontrivial and nonzero Floquet multipliers of the SOPS can  be calculated by solving the following boundary value problem:

\begin{theorem}[See {\cite[Theorem 2.2.3]{xie1991thesis}}]
	\label{prop:BVP}
	Suppose that $x\in \cX$ is a SOPS to  \eqref{eq:FeedBack} with period $L$.  
	Define the linearized DDE below:
	\begin{equation}
	y'(t) = - \alpha e^{x(t-1)} y(t-1). \label{eq:Linearized} 
	\end{equation}
	Then   $ \lambda \neq 0$ is a nontrivial eigenvalue of $ D_{\varphi} \Phi(x_0)$ if and only if \eqref{eq:Multiplier} has a nonzero solution $h \in H$ for which 
	\begin{equation}
	- y(L) \frac{x_L'(t)}{x'(L)} + y_L(t) = \lambda h(t) 
	\label{eq:Multiplier}
	\end{equation}
	where the function $h$ is then an eigenfunction of $ D_{\varphi} \Phi(x_0)$ associated with $\lambda$, and $ y(t)$ solves \eqref{eq:Linearized} with initial condition $ y_0 = h$. 
	
\end{theorem}

We are able to bound the Floquet multipliers by studying this  boundary value problem defined in~\eqref{eq:Linearized} and~\eqref{eq:Multiplier}, a calculation which is systematized through Algorithm \ref{alg:FloquetBound}. 
If this algorithm outputs a value $ \Lambda_{max} <1$  then all SOPS $x  \in  \kappa^{-1} ( K)$ are asymptotically stable. 
We are able to improve upon Xie's method in \cite{xie1991thesis,xie1993uniqueness} by repeating certain steps, somewhat analogous to the recursive bounds defined in Lemma \ref{prop:JonesMax}.  
The great advantage for doing this numerically as opposed to analytically is that these repetitions while tedious and time consuming for the mathematician   are ``effortless'' for the computer.

\begin{algorithm}
	\label{alg:FloquetBound}
Fix $I_\alpha = [ \alpha_{min} , \alpha_{max}] $ and $ K =  [q_{min}, q_{max} ] \times [ \bar{q}_{min}, \bar{q}_{max}] \times  [ M_{min}  , M_{max}] \subseteq \R^3$
with associated bounding functions $u_K,\ell_K$.  Furthermore, fix computational parameters $n_{Time},$ $N_{Floquet},$ $M_{Floquet} \in \N$. 
The output of the algorithm is $ \Lambda_{max} \in \R_+$.

	\begin{enumerate}
		\item
Define  $L_{min} \bydef q_{min} + \bar{q}_{min}$, $L_{max} \bydef q_{max} + \bar{q}_{max}$ and $I_L \bydef [ L_{min},L_{max}]$, and
		define the function $Y:[-1,0]\to \R$ by
		\begin{equation*}
		Y(t) \bydef 
				\begin{cases}
					1 & \mbox{ if } t \in [-1,0) \\
					0 & \mbox{ if } t =0 .
				\end{cases}
			\end{equation*}

		\item 
		\label{alg:step:Ybound}
		Extend the function $Y:[-1,L_{max}]\to \R$ by
		\begin{equation}
		Y(t) \bydef 
		\label{eq:AlgYbound}
			\alpha_{max} \int_0^t 
			\left(  Y(s-1)  \sup_{\ell_K \leq x \leq u_K} e^{x(s-1)} \right) ds 
			 \hskip .15\textwidth
			 \mbox{ if } t \geq 0,
		\end{equation}	
	evaluating the integral using an upper Riemann sum with a  uniform step size of  $1/ n_{Time}$. 
	Appendix \ref{sec:Appendix}  discusses in further details how we compute this integral.  
	
		\item
			Define $Z $ as below:
			\begin{equation*}
			Z(t) \bydef \left( \max_{L \in [L_{min},L_{max}]} Y(L)   \right)  
				\max_{\ell_K \leq x \leq u_K } 
			\left| 
				\frac{e^{x(t-1)}-1}{e^{x(-1)}-1} 
			\right|
			+ Y(t).
			\end{equation*}

		\item
			For $ t \in [-L_{min},0]$  define $Z_L$ as below:
			\begin{equation*}
			 Z_L(t)  \bydef \max_{ L_{min} \leq L \leq L_{max} }   Z(t+L).
			\end{equation*}
			
		\item
			For $t \in [-(L_{min}-1),0]$ refine the function $ Z_L$ by 
			\begin{eqnarray}
			Z'_L(-t)  &\bydef & 
			\alpha_{max} \int_{-t}^{0} 
			\left(
			Z_L(s-1)   \sup_{\ell_K \leq x \leq u_K} e^{x(s-1)} 
			\right) ds   	\label{eq:ZInt} \\
			Z_L(-t)  &\bydef &
			\min 
			\left\{ Z_L(-t) ,	Z'_L(-t) 
			\right\}, \nonumber
			\end{eqnarray}
			evaluating the integral using an upper Riemann sum with a  uniform step size of  $1/n_{Time}$.
			Appendix \ref{sec:Appendix}  discusses in further details how we compute this integral.  
			
			Repeat this step $M_{Floquet}$ number of times.

		\item
			Define 
			\begin{equation*}
				\Lambda_{max} \bydef \sup_{t \in [-1,0]} Z_L(t) .
			\end{equation*}

		\item
			If  $\Lambda_{max} < 1$ then \emph{STOP}.

		\item
			Otherwise define 
			\begin{equation}
			\label{eq:YRefContra}
				Y(t)  \bydef
				 \min 
				 \left\{ 
					 1 , Z_L(t)  
				 \right\}, \qquad   
				 \mbox{ for } t \in [-1,0]
			\end{equation}
			and GOTO Step \ref{alg:step:Ybound}.
			After reaching this step $ N_{Floquet}$ times, exit the program. 
			
	\end{enumerate}	  
	
\end{algorithm}

\begin{theorem}
	\label{prop:FloquetBound}
	Fix $I_\alpha = [ \alpha_{min} , \alpha_{max}]$ and  	$ K =  [q_{min}, q_{max} ] \times [ \bar{q}_{min}, \bar{q}_{max}] \times  [ M_{min}  , M_{max}] \subseteq \R^3$. 
		If Algorithm \ref{alg:FloquetBound} 
	terminates with $\Lambda_{max} <1$, then  all  SOPS $x \in \cX$ satisfying $ \kappa(x) \in K$ must be asymptotically stable.
	If the algorithm terminates having never reached Step 8, then the norm of all nontrivial Floquet multiplier are bounded above by $ \Lambda_{max}$. 	
\end{theorem}

\begin{proof}

	Fix some $ x \in \cX$ for which $\kappa(x) \in K$. By the definition made in Step 1, the period of $ x$ is some $L \in I_L$. 
	We use Theorem  \ref{prop:BVP} to estimate the range of Floquet multipliers of $x$. 
	That is, fix $\lambda \in \C$ and $h \in H$ and suppose that $ ( \lambda , h )$ is a solution to \eqref{eq:Multiplier}. 
	Define $y(t)$ to be the solution of~\eqref{eq:Linearized} through $h$, define $ z $ as 
\begin{equation}
z(t) \bydef - y(L) \frac{x'(t)}{x'(L)} + y(t) \
\label{eq:ZDef}
\end{equation}
and define $z_L(t) \bydef z(t+L)$. 
Hence $(\lambda,h)$ is a solution to~\eqref{eq:Multiplier} if and only if $ z_L(t) = \lambda h(t)$ for $ t \in [-1,0]$.  
As~\eqref{eq:Linearized} is a linear DDE, we may assume without loss of generality that $ \sup_{t \in [-1,0]} |h(t)| =1$. 
Thereby, it follows that 
\begin{equation}
| \lambda | = \sup_{t \in [-1,0]} | z_L(t)|.
\label{eq:LambdaZBound}
\end{equation}
If we can show that the RHS of~\eqref{eq:LambdaZBound} is less than $1$ uniformly for $ x \in \kappa^{-1} ( K)$, then we will have proven that all such SOPS are asymptotically stable. 
We prove that Steps 1-7 of Algorithm \ref{alg:FloquetBound}  produce functions $Y$, $Z$ and $Z_L$ and  a bound $\Lambda_{max}$ which satisfy the following inequalities uniformly for $x \in \kappa^{-1}(K)$
\begin{align*}
	|y(t) | &\leq Y(t), &
	|z(t) | & \leq  Z(t), &
	|z_L(t)|& \leq Z_L(t), &
	|\lambda| &\leq  \Lambda_{max}.
\end{align*}
We describe the results of each step of Algorithm \ref{alg:FloquetBound} in order, and then discuss how Step 8 affects what we may deduce about the output $\Lambda_{max}$.

\begin{enumerate}
	\item  By definition, if $h \in H $ then $ h(0) =0$, and by assumption $ |h(t)| \leq 1$ for $ t \in [-1,0]$.  Thereby our definition of $Y(t)$ in Step 1 satisfies $ |y(t) | \leq Y(t)$ for $t \in [-1,0]$.  
	
	\item 
	By definition $y$ solves the linear DDE in~\eqref{eq:Linearized}. 
	By variation of parameters it follows that 
	\begin{equation*}
y(t) = \int_0^t - \alpha e^{x(s-1)} y(s-1) ds
	\end{equation*}
	for all $ t \geq 0$. 
	Equation~\eqref{eq:AlgYbound} follows from this by taking a supremum over $ \alpha \in I_\alpha$ and $ \ell_K \leq x \leq u_K$. 
	Thereby, Step 2 produces a function $Y$ satisfying $ |y(t)| \leq Y(t)$ for $ t\geq 0$.

	\item
	Step 3 defines a function $Z$ to bound the norm of $z$ defined in~\eqref{eq:ZDef}. 
	 As $ x'(t) = - \alpha (e^{x(t-1)} -1)$  it follows that 
	 \[
	 |z(t)| \leq \left|y(L) \frac{e^{x(t-1)}-1}{e^{x(L-1)} -1} \right| + |y(t)|.
	 \]
	By periodicity, we may replace $ x'(L-1)$ with $ x'(-1)$.  
	By taking a supremum over $ L \in I_L$ and $ \ell_K \leq x \leq u_K$, it follows that the function defined in Step 3 satisfies $ |z(t)| \leq Z(t)$.  
	
	\item

	Since $L \in I_L$ and $ |z(t)| \leq Z(t)$, we obtain the estimate for $ t \in [ - L_{min} ,0 ]$ below:
		\[
			|z(t+L)|
			\leq Z(t+L) 
			\leq \max_{L_{min} \leq L \leq L_{max}} Z(t + L) 
			= Z_L(t).
		\]
	Since by definition $ z_L(t) = [z(t+L)]_0$, then for the definition of $ Z_L$ in Step 4, we have $|z_L(t)| \leq Z_L(t)$ for $ t \in [-1,0]$. 

	\item
	Note that both $y$ and $ x'$ satisfy~\eqref{eq:Linearized}, so by linearity $z$ solves~\eqref{eq:Linearized}.  
	Since $z_L(0) =0$, we obtain the following estimate using variation of parameters:
	\[
		z_L(-t) = - \int_{-t}^0 - \alpha e^{x(s-1)} z_L(s-1) ds .
	\]
	By taking the suprema over $\alpha \in I_\alpha $ and $ \ell_{K} \leq x \leq u_K$ as in Step 5, we obtain a refinement for which $ |z_L(t)| \leq Z_L(t)$.  
	This refinement can be repeated any number of times. 
	
	\item
	If $(\lambda,h)$ solves~\eqref{eq:Linearized}, then by~\eqref{eq:LambdaZBound} we obtain the following: 
\[
		|\lambda| = \sup_{t \in [-1,0]} |z_L(t)| 
		\leq \sup_{t \in [-1,0]} Z_L(t) 
		= \Lambda_{max}.
\]
	Hence $ |\lambda| < \Lambda_{max}$ uniformly for $ x \in \kappa^{-1} (K)$.

	\item
	We have shown that $ |\lambda| \leq  \Lambda_{max}$ for any Floquet multiplier $ \lambda$.  
	If $ \Lambda_{max} <1$, then it follows that $ x$ is asymptotically stable.

	\item
		If $ \Lambda_{max} \geq 1$, then we make the assumption that $x$ is not asymptotically stable for the sake of contradiction. 
		Then the largest Floquet multiplier $\lambda_{max}$ of $ x$ satisfies $ |\lambda_{max} | \in [ 1, \Lambda_{max}]$.  
		If $ h$ is an eigenfunction associated with $\lambda_{max}$, then $z_L(t) = \lambda_{max} \cdot h(t)$ for $ t \in [-1,0]$ and 
		furthermore $|h(t)| = |\lambda_{max}|^{-1} |z_L(t)| \leq |z_L(t)|$. 
		Hence for all $ t \in [-1,0]$ we may assume that the eigenfunction $h(t)$ satisfies the inequality: 
		\[
		|h(t)| \leq \min\{ 1 , |z_L(t)| \}.
		\]
		By definition  $y(t) = h(t)$ for $t \in [-1,0]$.  
		Hence for our refinement of $ Y$ in~\eqref{eq:YRefContra} it follows that $ | y (t) | \leq Y(t)$ for $ t \in [-1,0]$. 
\end{enumerate}

If the algorithm terminates having never passed through Step 8, then  $ |\lambda| \leq  \Lambda_{max} <1$ for all solutions $(\lambda,h)$ to~\eqref{eq:Multiplier} uniformly  for all SOPS $x \in \kappa^{-1}(K)$.  
If the program terminates having passed through Step 8 at least once, 
then it has shown that every solution $ (\lambda,h)$ to~\eqref{eq:Multiplier} satisfies $|\lambda|<1$ under the assumption that there exists a solution for which $ |\lambda| \geq 1$, a contradiction. 
In this case we have shown that $x$ is asymptotically stable without calculating an explicit bound on its Floquet multipliers.  \qedhere

\end{proof}


\section{A Comprehensive Algorithm}
\label{sec:CompResults}

We state our branch and prune algorithm in Algorithm \ref{alg:BranchAndPrune}, and describe how we use it to prove the uniqueness of SOPS to Wright's equation in Algorithm \ref{alg:Main}. 
 Algorithm \ref{alg:BranchAndPrune} takes as input an interval $ I_{\alpha } \subseteq \R$ and constructs an $I_\alpha$-exhaustive set. 
Furthermore, this algorithm uses several computational parameters:  $ \epsilon_1,\epsilon_2  \in \R$ which  defines  the algorithm's stopping criterion,  $ n_{Time} \in \N$ which defines the time resolution used in representing bounding functions on the computer, and  $N_{prune} \in \N$ which defines  the number of times the pruning algorithm is performed before branching.  
Additionally it requires the computational parameters $ i_0, j_0, N_{Period} \in \N$ needed for running Algorithms \ref{alg:InitialBoundsleq3} and \ref{alg:InitialBoundsgeq3}. 
As we have stated before, this is a canonical algorithm which terminates in finite time (see   \cite{scholz2011deterministic,ratschek1988new,horst2013global}).


\begin{algorithm}  \label{alg:BranchAndPrune}
	The input is an interval $ I_\alpha = [ \alpha_{min} , \alpha_{max}]$ and computational parameters $\epsilon_1,\epsilon_2>0$ and  $i_0,j_0, n_{Time} , N_{Period}, N_{Prune} \in \N$. 
	The output is a set $ \cA = \{ K_i : K_i \subseteq \R^3 \}$ 
	and an associated collection of bounding functions $ \{ u_K, \ell_K \}_{K \in \cA} $.

	\begin{enumerate}
		\item  Construct regions $K_1$ and $K_2$ according to Algorithms \ref{alg:InitialBoundsleq3} and \ref{alg:InitialBoundsgeq3} respectively. 
		Define the sets $ \cS = \{ K_1 , K_2\}$ and $ \cA = \emptyset$. 
		
		\item 	If $\cS = \emptyset$ then return $\cA$ and \emph{STOP}. 
		
		\item 	Define $K$ to be an element of $\cS$ and remove   $K$  from $ \cS$. 
		
		\item 	Define 	$\{K',u_{K'},\ell_{K'}\} $ to be the output of Algorithm \ref{alg:Prune} using input $K,u_K,\ell_K$ and computational parameter $n_{Time}$.  Then redefine $\{K,u_{K},\ell_{K}\} \bydef \{K',u_{K'},\ell_{K'}\} $.  	Repeat this step $N_{Prune}$ times.

		\item 	
		If the diameter of $K$ is less than $\epsilon_1$  and $\bar{q}<3$, or the diameter of $K$ is less than $ \epsilon_2$ and $\bar{q} \geq 3$, then add $K $ to $\cA$ and \emph{GOTO  Step 2}.
		
		\item 	
		Subdivide $K$  along its fattest dimension into two regions $K_A$ and $K_B$. 
		That is, write $ K = I_1 \times I_2 \times I_3$ where each $I_i$ is  given by the interval $ I_i = [a_i,b_i]$ and  
		 fix some $j \in \{ 1,2,3\}$ which maximizes $|b_j - a_j|$. 
		The regions $ K_A \bydef I_{1}' \times I_{2}' \times I_{3}'$ and $ K_B \bydef I_{1}'' \times I_{2}'' \times I_{3}''$ are defined according to the following formulas
		\begin{align}
		I_i' \bydef &
		\begin{cases}
		[a_i,b_i] & \mbox{ if } i \neq j \\
		[a_i, (a_i+b_i)/2] & \mbox{ if } i =j
		\end{cases} 
		& 
		I_i'' \bydef &
		\begin{cases}
		[a_i,b_i] & \mbox{ if } i \neq j \\
		[ (a_i+b_i)/2,b_i] & \mbox{ if } i=j.
		\end{cases}
		\end{align}

		\item Add to $ \cS$ the regions $K_A$ and $K_B$, each with associated bounding functions $u_K$ and $\ell_K$. Then GOTO  Step 2.

	\end{enumerate}
	
\end{algorithm}

\noindent
As a notational convention for the next two theorems we define $ \bigcup \cS \bydef  \bigcup_{K\in\cS} K$.

\begin{theorem}
	\label{prop:BranchNPrune}
	Fix an interval $I_\alpha = [ \alpha_{min} , \alpha_{max}]$ such that $\alpha_{min} > \pp$, and fix any  selection of computational parameters $\epsilon>0$ and  $i_0,j_0, n_{Time} , N_{Period}, N_{Prune} \in \N$. 
	If $ \cA$ is the output of Algorithm  \ref{alg:BranchAndPrune} with these inputs, then $\bigcup \cA$ is  $I_\alpha$-exhaustive.
	
\end{theorem}

\begin{proof}

We prove by induction that every time the algorithm arrives at Step 2, then $\bigcup \cS \cup  \bigcup \cA $ is $I_\alpha$-exhaustive. 
This suffices to prove the theorem, as the only way for the algorithm to exit is on Line 2 when $\cS = \emptyset$. 

For the initial case, 
the set $\bigcup \cS = K_1 \cup K_2$ produced in Step 1 is $I_\alpha$-exhaustive by  Proposition \ref{prop:ConstructInitialBounds}. 
The result of Step 3 simply rearranges the collection of regions, after which $\bigcup \cS \cup  \bigcup \cA \cup K$ is $I_\alpha$-exhaustive. 
In Step 4, this $I_\alpha$-exhaustivity  is maintained when replacing $K$ with the output of Algorithm \ref{alg:Prune} as a direct result of Proposition \ref{prop:Prune}.   
If Step 5 adds $K$ to $\cA$, then when the algorithm arrives at Step 2 the set $ \bigcup \cS \cup \bigcup \cA$ will be  $I_\alpha$-exhaustive. 
Otherwise Step 6 will divide $K$ into two regions $K_A$ and $K_B$ for which $ K = K_A \cup K_B$. 
Then in Step 7 both $K_A$ and $K_B$ are then added to $ \cS$, after which $ \bigcup \cS \cup \bigcup \cA$ is still $I_\alpha$-exhaustive.  \qedhere

\end{proof}

We are finally able to state our algorithm which can prove  that Wright's equation has a unique SOPS over a given range of parameters. 
\begin{algorithm}
	\label{alg:Main}
	The input is an interval $ I_\alpha = [ \alpha_{min} , \alpha_{max}]$ and  computational parameters $ \epsilon_1, \epsilon_2 >0$ and $i_0,j_0, n_{Time},N_{Period}, N_{Prune}, N_{Floquet}, M_{Floquet} \in \N$.  
	The output is a True or False statement.  
	
	\begin{enumerate}
		
		\item Run  Algorithm \ref{alg:BranchAndPrune}  with input $I_\alpha$ and computational parameters $ \epsilon_1,\epsilon_2$, $i_0$, $j_0$, $n_{Time}, N_{Period}$ and $N_{Prune}$. Define $ \cA$ and $\{ u_K,\ell_K\}_{K \in \cA}$ to be its output.

		\item For each $K  \in \cA$ calculate $ \Lambda_{max}(K) $ to be the output of Algorithm \ref{alg:FloquetBound}, run with input $I_\alpha$, $K$, $ u_K$, $\ell_K$, and computational parameters $n_{Time}, N_{Floquet}$ and $M_{Floquet}$.

		\item If $ \Lambda_{max}(K)  < 1 $ for all $ K \in \cA$, then return  \emph{TRUE}. 
		Otherwise return \emph{FALSE}. 
	\end{enumerate}
	
\end{algorithm}

\begin{theorem}
	\label{prop:Main}
	Fix an interval $ I_\alpha = [ \alpha_{min} , \alpha_{max}]$ with $\alpha_{min} > \pi/2$. 
	If Algorithm \ref{alg:Main} returns the output \emph{TRUE} for any selection of computational parameters $ \epsilon>0$, and $i_0,j_0,  n_{Time}, N_{Prune}, \\ N_{Floquet}, M_{Floquet} \in \N$, then  
	there exists a unique SOPS to Wright's equation for all $ \alpha \in I_\alpha$. 
\end{theorem}

\begin{proof}
	
	By Theorem \ref{prop:BranchNPrune} it follows that $ \bigcup \cA = \bigcup_{K\in \cA} K$ is an $I_\alpha$ exhaustive set. 
	That is, by Definition \ref{def:IalphaEx},  up to a time translation any SOPS to Wright's equation for parameter $\alpha \in I_\alpha$ can be expressed as a function $ x \in \cX$ for which $ \kappa (x) \in \bigcup \cA$. 
	If Algorithm \ref{alg:FloquetBound} 
	terminates with $\Lambda_{max}(K) <1$ for all $ K \in \cA$, then by Theorem \ref{prop:FloquetBound} it follows that any SOPS $x \in \cX$ satisfying $ \kappa(x) \in \bigcup \cA $ must be asymptotically stable. 
	Hence, by Theorem \ref{prop:Xie} it follows that there must be a unique SOPS to Wright's equation for each $\alpha \in I_\alpha$.  
\end{proof}

\section{Discussion}
\label{sec:Discussion}

In Algorithm \ref{alg:Main} we defined an algorithm which, if successful, proves the uniqueness of SOPS to Wright's equation for a finite range of parameters $ I_\alpha$. 
Below we describe how we applied this algorithm to prove Theorem \ref{prop:MainResult}.

\begin{proof}[Proof of Theorem \ref{prop:MainResult}]

	To prove  Theorem \ref{prop:MainResult} we divide the interval $[1.9,6.0]$ into various  subintervals $I_\alpha$, and then divide each of these intervals into further subintervals of width $ \Delta \alpha$. 
	For example, the interval $I_{\alpha} = [2.1,6.0]$ with $ \Delta \alpha = 0.1$  was divided into   subintervals   $[2.1,2.2]$, $[2.2,2.3]$, \dots , $[5.9,6.0]$. 
	The various computational parameters we used are given in the table below (see \cite{webpage_jones_conjecture} for associated MATLAB code). 

\begin{center}

$
	\begin{array}{ c |c c c c c c c c c c}
		I_\alpha 		& \Delta \alpha & n_{Time} & \epsilon_1 & \epsilon_2 & i_0 & j_0 & N_{Period} &  N_{Prune} & N_{Floquet} & M_{Floquet} \\ 
	\hline
		\,[1.90,1.96] 	& 0.01 & 128 & 0.02 & 0.25 & 2 & 20  & 10 & 4 & 20 & 5 \\
		\,[1.96,2.10] 	& 0.01 &  64 & 0.05 & 0.25 & 2 & 20  & 10 & 4 & 20 & 5 \\
		\,[2.10,6.00]	& 0.10 &  32 & 0.05 & 0.25 & 2 & 20  & 10 & 4 & 20 & 5 \\	
	\end{array}
$
	\captionof{table}{For descriptions of how these parameters affect Algorithm \ref{alg:Main}, refer to Algorithms \ref{alg:InitialBoundsleq3} and \ref{alg:InitialBoundsgeq3} for $i_0$, $j_0$ and $N_{Prune}$;  refer to Algorithm \ref{alg:FloquetBound} for $N_{Floquet}$ and $M_{Floquet}$; and refer to Algorithm \ref{alg:BranchAndPrune} for $\epsilon_1, \epsilon_2$ and  $N_{Prune}$.  }
\end{center}	
	For each of these parameter values, we ran Algorithm \ref{alg:Main} which returned \emph{TRUE} as its output. 
	By Theorem \ref{prop:Main} it follows that there must be a unique SOPS to Wright's equation for each $\alpha \in [1.9,6.0]$. 	
\end{proof}

As described in Theorem \ref{prop:FloquetBound}, if Algorithm \ref{alg:FloquetBound} terminates without having reached Step 8, then it produces explicit bounds on the Floquet multipliers of the SOPS to Wright's equation.  
These bounds are summarized in Figure \ref{fig:FloquetMultipliers}. 
In the range $ [2.2,6.0]$ Algorithm \ref{alg:FloquetBound} exits on Step 7, so by Theorem \ref{prop:FloquetBound} we obtain an upper bounds on the Floquet multipliers.  
In the regime $ \alpha \in [1.90 , 2.20]$  Algorithm \ref{alg:FloquetBound} only terminated after reaching Step 8 at least once, so we are only able to deduce that any non-trivial Floquet multiplier has modulus strictly bounded above by $1$.  
In total, the computation took 115 hours to run using a i7-5500U processor, and Algorithm \ref{alg:BranchAndPrune} accounted for 94\% of the computation time.

Running Algorithm \ref{alg:Main} at high values of $\alpha$ is computationally expensive.  
This is because the period length of SOPS to Wright's equation grows exponentially  \cite{nussbaum1982asymptotic}, whereby our algorithm's run time and memory requirements also increases  exponentially in $\alpha$.  
Nevertheless, proving Theorem \ref{prop:MainResult} with an upper limit of $ \alpha = 6$ is sufficient for our purposes considering the results in \cite{xie1991thesis} proved uniqueness for $ \alpha \geq 5.67$. 

\begin{figure}[h]
	\centering
	\includegraphics[width = 1\textwidth]{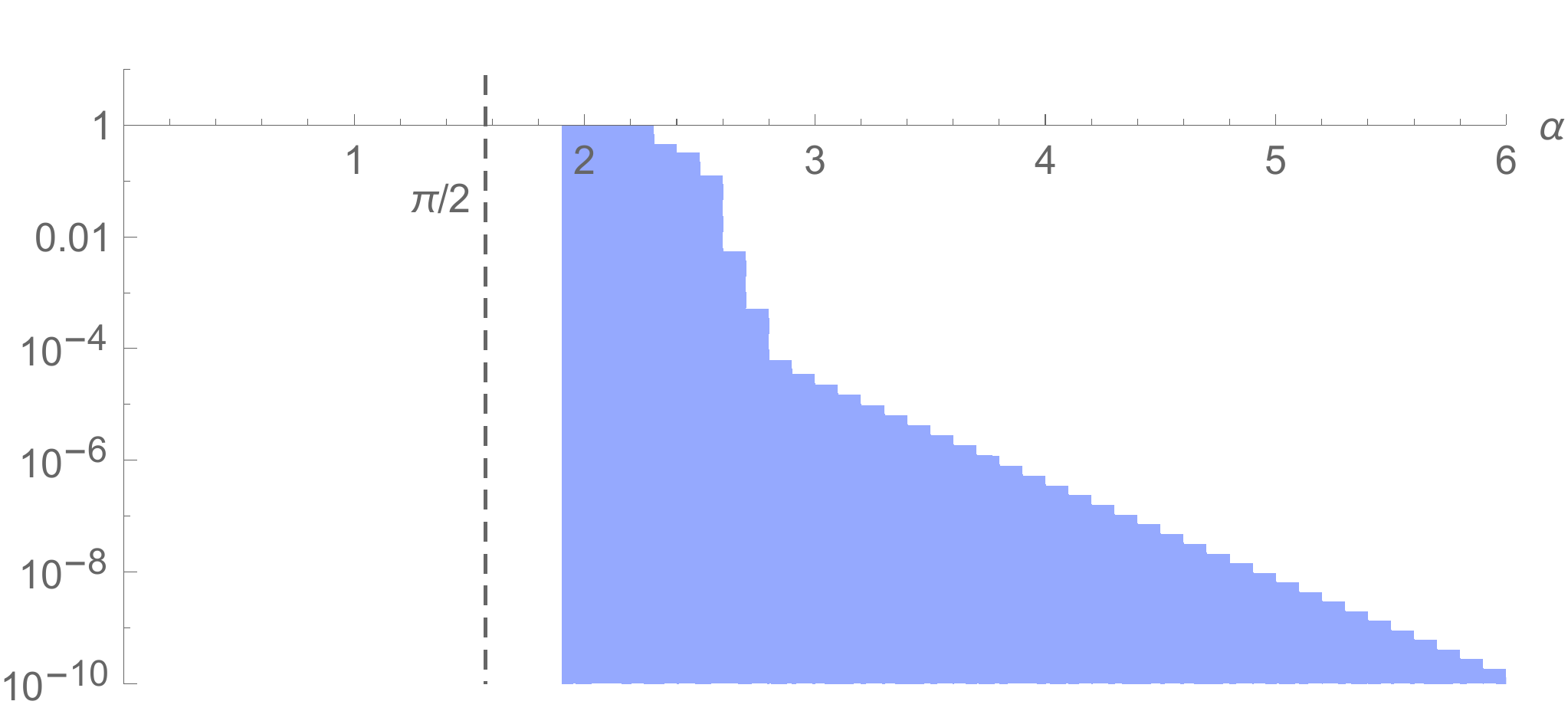}
	\caption{ An upper bound on the modulus of the Floquet multipliers for SOPS to Wright's equation for $ \alpha \in [1.9,6.0]$.}
	\label{fig:FloquetMultipliers}
\end{figure}

A different  challenge  presents itself for decreasing the lower limit of $\alpha = 1.9$ in Theorem \ref{prop:MainResult}. 
Namely, Xie's method for bounding the largest Floquet multiplier is not well suited to weakly attracting SOPS. 
Even when using precise numerical approximations (from \cite{lessard2010recent}) of SOPS to Wright's equation at  single values of $\alpha$, Algorithm \ref{alg:FloquetBound} was only able to show that the SOPS was asymptotically stable  for values of $\alpha$ no lower than $ 1.85$.  
By decreasing the parameters $ \Delta \alpha$  and $\epsilon_1$, and increasing the other computational parameters, we could expect the uniqueness result for $ \alpha \geq 1.9$ could be pushed closer to $ \alpha =1.85$.  
However we believe that new ideas are required in order to prove Conjecture \ref{prop:RemainingConjecture}.

\begin{footnotesize}


\end{footnotesize}

\appendix

\section{Appendix: Computational Considerations}
\label{sec:Appendix}

Interval arithmetic (e.g. see \cite{moore1966interval}) is an indispensable tool for producing computer-assisted proofs in nonlinear dynamics (e.g. see \cite{MR2652784,MR2807595,MR1420838}). 
As the name suggests, interval arithmetic extends arithmetic operations to intervals of the real numbers, such as:
\begin{align*}
	[a,b] + [c,d] \bydef& [a+c,b+d] & 	[a,b] - [c,d] \bydef& [a-d,b-c]. & 
\end{align*}
In this manner, if $x_1 \in [a,b]$ and $ x_2  \in [c,d]$ then $ x_1 + x_2 \in 	[a,b] + [c,d] $. 
Moreover,   for any function $ f: \R^n \to \R$ and its interval extension $ F$ we have the relation:
\[
f(x_1,x_2, \dots, x_n ) \in 
F\left( [a_1,b_1],[a_2,b_2], \dots , [a_n,b_n] \right) 
\]
for all $ x_i \in [a_i, b_i]$. 

When interval arithmetic is performed on a computer the endpoints are necessarily binary numbers, and outward rounding is used to ensure proper enclosure. 
This allows us to compute a verified enclosure of the value of a function on numbers not representable on a binary computer, such as $ \tfrac{1}{3}$. 
Furthermore this enables us to easily compute upper and lower bounds on a function over a rectangular domain of any size. 
While these bounds may not necessarily be sharp, they are guaranteed to be accurate.

To implement our algorithm we used \emph{Intlab}: an interval arithmetic package for Matlab \cite{rump1999intlab}. 
Some of the calculations we performed are a simple application of interval arithmetic, such as defining $ I_q, I_{\bar{q}}, I_M$ in Algorithm \ref{alg:InitialBoundsleq3}.  
However there is a nontrivial degree of complexity in how we store and represent the functions used in the algorithms, such as $u_K ,\ell_K$ in Algorithm \ref{alg:Prune} or $ Y,Z,Z_L$ in Algorithm \ref{alg:FloquetBound}. 
In a word, we defined these functions to be piecewise constant.

To explain our methodology, first fix a constant $n_{Time} \in \N$. 
To define an interval extension of a function $y: \R \to \R$, we  define a collection of intervals $ I_i^P , I_i^I \subseteq \R$ for $ i \in \Z$ and define $Y$ as follows:
\[
Y(t) = 
\begin{cases}
I_{i}^P  & \mbox{ if }  t = \tfrac{i}{n_{Time}} \\
I_{i}^I  & \mbox{ if }  t \in \left(  \tfrac{i}{n_{Time}} ,\tfrac{i+1}{n_{Time}}\right).
\end{cases}
\]
Of course any computer has finite memory, and so we would only store the function $ Y$ over a finite domain. 
Furthermore, as the bounding functions $ u ,\ell$ are intended to provide upper and lower bounds on a function $x$, we simply define an interval valued  function $X(t) = [ \ell(t) , u(t)]$.  
In Figure \ref{fig:FunctionStorage} we present a graphical representation of how we store such a function, wherein we have defined the function $X(t) $ for $t \in [-1,0]$ as follows: 
\begin{align*}
I_{-4}^P \bydef& [-2.0,-1.2]  &		I_{-4}^I \bydef& [-2.0,-0.9]  \\
I_{-3}^P \bydef& [-1.6,-0.9]  &		I_{-3}^I \bydef& [-1.6,-0.6]  \\
I_{-2}^P \bydef& [-1.2,-0.6]  &		I_{-2}^I \bydef& [-1.2,-0.3]  \\
I_{-1}^P \bydef& [-0.8,-0.3]  &		I_{-1}^I \bydef& [-0.8,-0.0]  \\
I_{0}^P  \bydef& [0.0,0.0]  &		
\end{align*}

For such functions, it is a straightforward procedure to calculate its supremum.  To calculate $\sup_{t \in [a,b]} x(t)$ one simply needs to compare the intervals $I_{i}^P$ for which $ a \leq \frac{i}{n_{Time}} \leq b$, and the intervals $I_{i}^I$ for which $ a-n_{Time}^{-1} < \frac{i}{n_{Time}} < b$. 
Both these collections of intervals are finite.  
For bounds which are defined to be the integrals of various functions, as in~\eqref{eq:Udef}
 and~\eqref{eq:Ldef} of Lemma \ref{prop:PeriodBound}, we use a Riemann sum of step size $1 / n_{Time}$.

Unfortunately there is a loss in fidelity when we numerically integrate these functions, as we do in Step 2 of Algorithm \ref{alg:Prune}. 
Therein we refine the values of  $u_{K'}(t_0+s),$ $\ell_{K'}(t_0+s),$ $u_{K'}(t_0-s)$ and  $\ell_{K'}(t_0-s)$,    where $t_0 = \tfrac{i_0}{n_{Time}}$ and $ s \in [0, \tfrac{1}{n_{Time}}]$.   
To just discuss the refinements of $ u_{K'}(t_0 + s)$ and $ \ell_{K'}(t_0 +s)$, if we choose $ s =  \tfrac{1}{n_{Time}}$, then this procedure refines the bound of $[\ell_{K'}(\frac{i_0+1}{n_{Time}}),  u_{K'}(\frac{i_0+1}{n_{Time}})]$, a value  which is stored in the interval $I_{i_0+1}^P$.  
However in order to refine $ I_{i_0}^I$ this interval must include $[\ell_{K'}(t_0 +s'),  u_{K'}(t_0 +s')]$ for all $ s' \in ( t_0,t_0 + \tfrac{1}{n_{Time}})$. 
This is represented in Figure \ref{fig:FunctionStorage}, where the darker red region represents the sharpest possible bounds able to be derived from in Step 2 of Algorithm \ref{alg:Prune} when integrating the initial data given above, and the pink region represents the values we store in the computer. 
When we define functions as integrals as in Steps 2 and 5 of Algorithm \ref{alg:FloquetBound} we use the same procedure.

\begin{figure}[h]
	\centering
	\includegraphics[width = 1.0\textwidth]{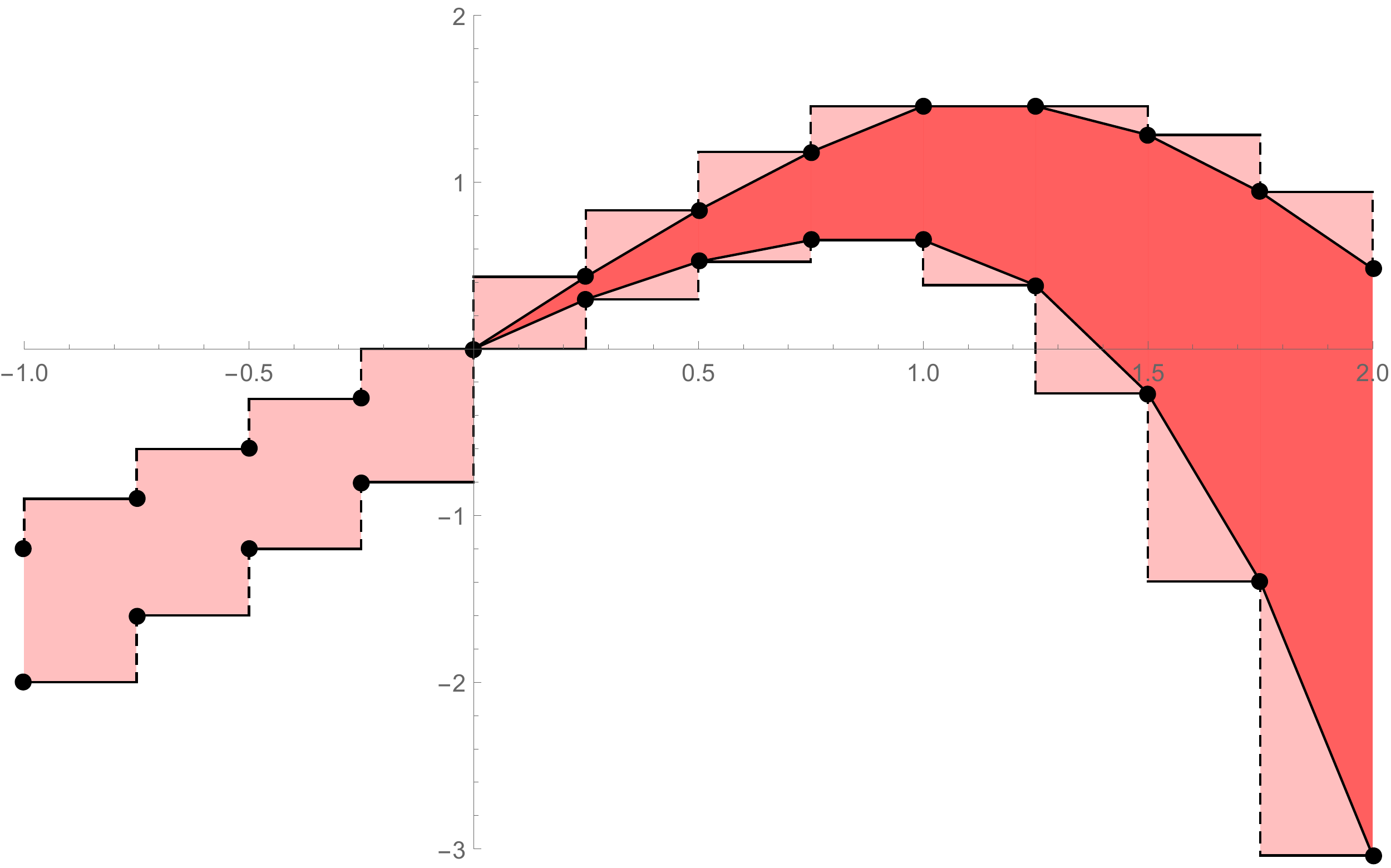}
	\caption{ An example of how we store an interval valued function $[ \ell(t) , u(t) ]$  in our algorithm.}
	\label{fig:FunctionStorage}
\end{figure}

\end{document}